\documentclass[a4,10pt]{amsart}

\usepackage{ulem}
\usepackage[dvipdfmx]{graphicx,xcolor} %{\color{red} }
\usepackage{amscd} %\begin{CD}
\usepackage{txfonts}

\newtheorem{thm}{Theorem}[section]
\newtheorem{lem}[thm]{Lemma}
\newtheorem{prop}[thm]{Proposition}
\newtheorem{cor}[thm]{Corollary}

\theoremstyle{definition}

\newtheorem{example}[thm]{Example}

\theoremstyle{rem}
\newtheorem{rem}[thm]{Remark}

\numberwithin{equation}{section}

\makeatletter
\@namedef{subjclassname@2020}{\textup{2020} Mathematics Subject Classification}
\makeatother

\begin{document}

\title[B\"{o}ttcher coordinates]
{Newton polygons and B\"{o}ttcher coordinates near infinity for polynomial skew products}  

\author[K. Ueno]{Kohei Ueno}
\address{Daido University, Nagoya 457-8530, Japan}
%\curraddr{}
\email{k-ueno@daido-it.ac.jp}
%\thanks{This work was supported by the Research Institute for Mathematical Sciences, 
%           a Joint Usage/Research Center located in Kyoto University.}

%32B10 Germs of analytic sets
%37F50 Small divisors, rotation domains and linearization; Fatou and Julia sets

\date{}
\dedicatory{}

\begin{abstract}
Let $f(z,w)=(p(z),q(z,w))$ be a polynomial skew product such that 
the degrees of $p$ and $q$ are grater than or equal to $2$.
Under one or two conditions, % assumptions,
we prove that $f$ is conjugate to a monomial map
on an invariant region near infinity.
The monomial map and the region are determined by
the degree of $p$ and a Newton polygon of $q$.
Moreover,
%$f$ extends to a rational map on the projective space or a weighted projective space, and 
the region is included in the attracting basin of 
a superattracting fixed or indeterminacy point at infinity,
or in the closure of the attracting basins of two point at infinity.
\end{abstract}

\maketitle

%%%%%%%%%%%%%%%%%%%%%%%%%%%%%%%%%%%%%%%%%%%%%%%%%%%%%%%%%%%%%%%%%%%%%%%%%%%%
%%%%%%%%%%%%%%%%%%%%%%%%%%%%%%%%%%%%%%%%%%%%%%%%%%%%%%%%%%%%%%%%%%%%%%%%%%%%
\section{Introduction}

%%%%%%%%%%%%%%%%%%%%%%%%%%%%%%%%%%%%%%%%%%%%%%%%%%%%%%%%%
\subsection{Background}

Thanks to B\"{o}ttcher's theorem \cite{b},
the local dynamics around the superattracting fixed point
of a holomorphic germ in dimension 1 is completely well understood:
the germ is conjugate to its lowest degree term
on a neighborhood of the point.
We can apply this theorem to polynomials.
A polynomial of degree grater than or equal to $2$
extends to a holomorphic map on the Riemann sphere
with a superattracting fixed point at infinity, and so
it is conjugate to its highest degree term
on a neighborhood of the infinity.
These changes of coordinate are called 
the B\"{o}ttcher coordinates for the germ at the fixed point and 
for the polynomial at infinity, and
derives dynamically nice subharmonic functions 
on the attracting basin of the point and on $\mathbb{C}$,
respectively.

B\"{o}ttcher's theorem does not extend to higher dimensions entirely.
As pointed out in \cite{hp}, 
the complexity of the critical orbit of the germ is an obstruction.
Whereas the superattracting fixed point of a holomorphic germ $p$ in dimension 1
is an isolated critical point of $p$ and forward invariant under $p$,
the superattracting fixed point of a holomorphic germ $f$ in dimension 2 
is contained in the critical set of $f$,
which may not be forward invariant under $f$.
The case of polynomial maps has more difficulties.
Although a polynomial map on $\mathbb{C}^2$ extends to 
a rational map on the projective space $\mathbb{P}^2$,
it may not be holomorphic and, moreover,
we have to add the line at infinity to $\mathbb{C}^2$,
instead of the point at infinity.
Other compactifications of $\mathbb{C}^2$
have similar difficulties.

Favre and Jonsson \cite{fj SA} studied and gave general theorems 
for both cases in dimension 2.
For superattracting holomorphic germs,
they gave normal forms on regions 
whose closure contains the superattracting fixed point
in Theorems C and 5.1
by using the rigidifications.
Moreover,
using these normal forms,
they investigated the attraction rates and
constructed dynamically nice plurisubharmonic functions 
defined on the attracting basins.
For polynomial maps on $\mathbb{C}^2$,
they gave normal forms on regions near infinity in Theorem 7.7,
investigated the degree growths and
constructed dynamically nice plurisubharmonic functions 
defined on $\mathbb{C}^2$,
assuming that the maps are not conjugate to skew products.
Moreover,
they advanced their study on the dynamics of polynomial maps in \cite{fj poly}.
In particular,
one can find statements on normal forms
in Theorem 3.1 and in Section 5.3.

We are interested in the dynamics of skew products.
A skew product is a germ or map in dimension 2
of the form $f(z,w) = (p(z), q(z,w))$.
See \cite{j} and \cite{fg} for fundamental studies of polynomial skew products.
Let $f$ be a holomorphic skew product germ with a superattracting fixed point at the origin.
Under one or two conditions, % assumptions,
we \cite{ueno SA} have succeeded in constructing 
a B\"{o}ttcher coordinate for $f$ concretely on an invariant region
whose closure contains the origin, % superattracting fixed point,
which conjugates $f$ to a monomial map.
The original idea in \cite{ueno SA} and in our other previous studies is
to assign a suitable weight.
The monomial map and the region are determined by 
the order of $p$ and the Newton polygon of $q$.
Using the same ideas and results as in \cite{ueno SA},
we investigated the attraction rates on the vertical direction in \cite{ueno attr rates} and
derived plurisubharmonic functions from B\"{o}ttcher coordinate in \cite{ueno Green funs},
which describe the vertical dynamics well and
% some of which are different from those in \cite{fj SA}.
some of which do not appear in \cite{fj SA}.
 
In this paper
we adapt the same ideas as in \cite{ueno SA} 
to the case of polynomial skew products.
Let $f$ be a polynomial skew product on $\mathbb{C}^2$.
Under one or two conditions, % assumptions,
we construct a B\"{o}ttcher coordinate for $f$ concretely 
on an invariant region near infinity,
which conjugates $f$ to a monomial map.
The monomial map and the region are determined by 
the degree of $p$ and a Newton polygon of $q$.
Here the definition of a Newton polygon is 
different or opposite from the usual one.
The map $f$ extends to the rational map 
on the projective space or a weighted projective space, and
the region is included in the attracting basin of 
a superattracting fixed or indeterminacy point at infinity, or
in the closure of the attracting basin of two points at infinity.
This result completes our previous study in \cite{ueno poly}
and gives a well organized consequence. 
We expect that the ideas and results in this paper are useful 
to investigate the attraction rates on the vertical direction and
to derive plurisubharmonic functions
which describe the vertical dynamics well.

%%%%%%%%%%%%%%%%%%%%%%%%%%%%%%%%%%%%%%%%%%%%%%%%%%%%%%%%%
\subsection{Main results}

Let us state our main results precisely.
Let $f$ be a polynomial skew product on $\mathbb{C}^2$
of the form $f(z,w) = (p(z),q(z,w))$,
where $\deg p = \delta \geq 2$ and $\deg q \geq 2$.
Then we may write 
$p(z) = a_{\delta} z^{\delta} + o(z^{\delta})$,
where $a_{\delta} \neq 0$, and
$q(z,w) = \sum_{i, \, j \geq 0} b_{ij} z^{i} w^{j}$.
It is clear that the dominant term of $p$ is $a_{\delta} z^{\delta}$.
On the other hand,
we can find a `dominant' term $b_{\gamma d} z^{\gamma} w^d$ of $q$
by making use of the degree of $p$ and a Newton polygon of $q$; % by ? with ? by using ?
thus
\[
p(z) = a_{\delta} z^{\delta} + o(z^{\delta})
\text{ and }
q(z,w) = b_{\gamma d} z^{\gamma} w^d + \sum_{(i, j) \neq (\gamma, d)} b_{ij} z^{i} w^{j}.
\]
More precisely,
$b_{\gamma d} z^{\gamma} w^d$ is dominant on an region
$U = \{ |z|^{l_1 + l_2} > R^{l_2} |w|, |w| > R |z|^{l_1} \}$ 
for rational numbers $0 \leq l_1 < \infty$ and $0 < l_2 \leq \infty$,
which are also determined by 
the degree of $p$ and a Newton polygon of $q$ and
called weights in \cite{ueno poly} and \cite{ueno SA}. % the previous papers 

We define the Newton polygon $N(q)$ of $q$ as 
the convex hull of the union of $D(i,j)$ with $b_{ij} \neq 0$,
where $D(i,j) = \{ (x,y) \ | \ x \leq i, y \leq j \}$. 
This definition is different or opposite from the usual one.
Let $(n_1, m_1)$, $(n_2, m_2)$, $\cdots, (n_s, m_s)$ be the vertices of $N(q)$,
where $n_1 < n_2 < \cdots < n_s$ and $m_1 > m_2 > \cdots > m_s$.
Let $T_k$
be the $y$-intercept of the line $L_k$ passing through 
the vertices $(n_k, m_k)$ and $(n_{k+1}, m_{k+1})$
for each $1 \leq k \leq s-1$.

\vspace{2mm}
\begin{itemize}
\item[Case 1] 
If $s = 1$, then $N(q)$ has the only one vertex, which is denoted by $(\gamma, d)$. \\
For this case, we define
$l_1 =  l_2^{-1} = 0$ and so $U = \{ |z| > R, |w| > R \}$.
\end{itemize}
\vspace{2mm}

Difficulties appear when $s > 1$,
which is divided into the following three cases. 

\vspace{2mm}
\begin{itemize}
\item[Case 2] 
  If $s > 1$ and $\delta \leq T_{1}$, then  we define
  \[
  (\gamma, d) = (n_1, m_1), \ 
  l_1 = \frac{n_2 - n_1}{m_1 - m_2}
  \text{ and } l_2^{-1} = 0.
  \] 
  Hence $U = \{ |z| > R, |w| > R |z|^{l_1} \}$. 
  \vspace{4mm}
\item[Case 3] 
  If $s > 1$ and $T_{s-1} \leq \delta$, then  we define
  \[
  (\gamma, d) = (n_{s}, m_{s}), \
  l_1 = 0 
  \text{ and } l_2 = \frac{n_s - n_{s-1}}{m_{s-1} - m_s}.
  \] 
  Hence $U = \{ |z|^{l_2} > R^{l_2} |w|, |w| > R \} = \{ R < |w| < R^{-l_2} |z|^{l_2} \}$.  
  \vspace{4mm}
\item[Case 4]
  If $s > 2$ and $T_{k-1} \leq \delta \leq T_{k}$ for some $2 \leq k \leq s-1$, then  we define
  \[
  (\gamma, d) = (n_k, m_k), \ 
  l_1 = \frac{n_{k+1} - n_k}{m_k - m_{k+1}} 
  \text{ and } l_1 + l_2 = \frac{n_k - n_{k-1}}{m_{k-1} - m_k}. 
  \] 
  Hence $U = \{ R |z|^{l_1} < |w| < R^{-l_2} |z|^{l_1 + l_2} \}$.
\end{itemize}
\vspace{2mm}

% Note that $\gamma \geq 1$ if $d = 1$ for Cases 1 and 2 and
% that $\delta > d$ and $\gamma > 0$ for Cases 3 and 4.

Let $f_0(z,w) = (p_0(z), q_0(z,w)) = (a_{\delta} z^{\delta}, b_{\gamma d} z^{\gamma} w^d)$.
%By the definition, we have the following.

\begin{prop}\label{main lemma}
If $d \geq 2$ 
or if $d = 1$ and $\delta \neq T_k$ for any $k$, 
then
\begin{enumerate}
\item for any small $\varepsilon >0$, 
there is $R >0$ such that
$|p - p_0| < \varepsilon |p_0|$ and 
$|q - q_0| < \varepsilon |q_0|$ on $U$, and 
\item $f(U) \subset U$ for large $R >0$.
\end{enumerate}
\end{prop}

This proposition induces a conjugacy on $U$ from $f$ to $f_0$ 
as in the one dimensional case.

\begin{thm}\label{main theorem}
If $d \geq 2$ or if $d = 1$ and $\delta \neq T_k$ for any $k$, 
then there is a biholomorphic map $\phi$ defined on $U$
that conjugates $f$ to $f_0$ for large $R >0$.
Moreover,
for any small $\varepsilon > 0$, 
there is $R > 0$ such that 
$|\phi_1 - z| < \varepsilon |z|$ and 
$|\phi_2 - w| < \varepsilon |w|$ on $U$,
where $\phi = (\phi_1, \phi_2)$.
\end{thm}

We call $\phi$ the B\"{o}ttcher coordinate for $f$ on $U$ and
construct it as the limit of the compositions of $f_0^{-n}$ and $f^{n}$,
where the branch of $f_0^{-n}$ is taken as $f_0^{-n} \circ f_0^n = id$.

\begin{rem}[Two dominant terms]
If $s > 1$ and $\delta = T_k$ for some $1 \leq k \leq s-1$,
then there are two different `dominant' terms of $q$.
% then there are two different ``dominant'' terms of $q$.
Moreover,
if both satisfy the degree condition,
then there are two disjoint invariant regions 
on which $f$ is conjugate to each of the two different monomial maps.
\end{rem}

\begin{rem}[Comparision with our previous results]
We proved the main results for Cases 1 and 2 in \cite{ueno poly}.
More strongly,
we can sometimes enlarge $U$ as proved in \cite{ueno poly}.
For Case 1, 
the same results hold on $U = \{ |z| > R, |w| > R |z|^{l_1^*} \}$
if $l_1^*$ is well defined,
where $l_1^*$ is a non-positive rational number and relates to $l_1$ for Case 2;
see Remark \ref{larger regions for case 1} for details.
Moreover,
for Cases 1 and 2,
the same results hold on $U = \{ |w| > R^{1+l_1^*}, |w| > R |z|^{l_1^*} \}$
and $U = \{ |w| > R^{1+l_1}, |w| > R |z|^{l_1} \}$
if $\gamma = 0$, respectively.
\end{rem}

\begin{rem}[Uniqueness]
It is known that a B\"{o}ttcher coordinate for $p$ is unique up to 
multiplication by an $(\delta - 1)$st root of unity.
A similar uniqueness statement holds
for Cases 1 and 2 with some conditions;
see Proposition 4 in \cite{ueno poly}.
\end{rem}

%%%%%%%%%%%%%%%%%%%%%%%%%%%%%%%%%%%%%%%%%%%%%%%%%%%%%%%%%
\subsection{Organization}

We first prove Proposition \ref{main lemma} and
illustrate the main results in term of blow-ups % explain 
when $l_1$ and $l_2^{-1}$ are integer
for Cases 2, 3 and 4
in Sections 2, 3 and 4, respectively,
by the same strategy as in \cite{ueno SA}.
Although Case 2 was already proved in \cite{ueno poly},
we provide uniform presentations 
in terms of Newton polygons and blow-ups.
The proofs of the main results for Case 1
are similar to and simpler than the other cases.
%%% 
% We then explain the importance of the weights $l_1$ and $l_2$ 
% from a different viewpoint in Section 5. In fact,
We then introduce intervals of real numbers for each of which the main results hold
in Section 5;
the intervals contain $l_1$ and $l_2$ as important numbers.
Moreover,
we associate rational numbers in the intervals to formal branched coverings of $f$,
which are a generalization of the blow-ups,
% and consider when the covering is well defined. 
and give sufficient conditions for the coverings to be well defined. 
% This section can be considered as an appendix, 
% and one may skip for the proofs of the main results.
%%%
Rational extensions of $f$ to the projective space and weighted projective spaces
are dealt with in Section 6. % and we illustrate that
% the region $U$ is included in the attracting basin of 
% a superattracting fixed point or an indeterminacy point at infinity,
% or in the closure of the attracting basins of two point at infinity.
In Sections 5 and 6,
besides $l_1$ and $l_2$,
the weight $\alpha_0 = \gamma/ (\delta - d)$ plays an important role
when $\delta \neq d$.
One may skip Sections 5 and 6 for the proof of the main theorem.

%%%
We next prove Theorem \ref{main theorem} in Section 7:
% by the same arguments as in \cite{ueno SA}:
it follows from Proposition \ref{main lemma} that
the composition $\phi_n = f_0^{-n} \circ f^n$ is well defined on $U$,
converges uniformly to $\phi$ on $U$,
and the limit $\phi$ is biholomorphic on $U$.
The proof of the uniform convergence of $\phi_n$ is different
whether $d \geq 2$ or $d = 1$.
We use Rouch\'e's Theorem to obtain the injectivity of $\phi$.
% and the $r$ from Lemma \ref{main lemma} may need to be decreased.
%%%
The extension problem of $\phi$ is dealt with in Section 8.
Roughly speaking,
$\phi$ extends by analytic continuation 
until it meets the critical set of $f$.
%%%
% In Section 9 we rewrite a statement in \cite{ueno poly}
% on the uniqueness of a B\"{o}ttcher coordinate.
% It is known that a B\"{o}ttcher coordinate for $p$ is unique up to 
% multiplication by an $(\delta - 1)$st root of unity.
% A similar uniqueness statement holds
% for Cases 1 and 2 with two reasonable conditions if $d \geq 2$.
%%%
Finally, other changes of coordinate derived from $\phi$
are shown in Section 9.

The results in Sections 6, 7, 8 and 9 are obtained by 
almost the same or similar arguments as in \cite{ueno poly} and \cite{ueno SA}:
% Roughly speaking,
we mainly refer \cite{ueno poly} for Sections 6 and 9 and
\cite{ueno SA} for Sections 7 and 8, respectively.
We mainly use the same notations as in \cite{ueno SA} in this paper.
% some of which are different from those in \cite{ueno poly}.

%%%%%%%%%%%%%%%%%%%%%%%%%%%%%%%%%%%%%%%%%%%%%%%%%%%%%%%%%%%%%%%%%%%%%%%%%%%%
%%%%%%%%%%%%%%%%%%%%%%%%%%%%%%%%%%%%%%%%%%%%%%%%%%%%%%%%%%%%%%%%%%%%%%%%%%%%
\section{Main proposition and Blow-ups for Case 2}

We prove Proposition \ref{main lemma} for Case 2 in this section.
Let $s > 1$, 
\[
\delta \leq T_{1}, \ 
(\gamma, d) = (n_1, m_1), \ 
l_1 = \frac{n_2 - n_1}{m_1 - m_2}
\text{ and } l_2^{-1} = 0.
\] 
% Note that 
Then $d \geq 1$, and $\gamma \geq 1$ if $d = 1$. % by the setting.
We first prove Proposition \ref{main lemma} in Section 2.1
and then illustrate our main results in terms of blow-ups when $l_1$ is integer in Section 2.2.

We assume that $a_{\delta} = 1$ and $b_{\gamma d} = 1$
for simplicity % in the proofs and explanations
through out the paper.
% note that we may assume that $a_{\delta} = 1$ and $b_{\gamma d} = 1$ if $d \geq 2$
% without loss of generality, by taking an affine conjugate.
Let us denote $f \sim f_0$ on $U$ as $R \to \infty$ for short
if $f$ satisfies the former statement in Proposition \ref{main lemma}.
%for any small $\varepsilon$ there is $R$ such that
%$|p - p_0| < \varepsilon |p_0|$ and $|q - q_0| < \varepsilon |q_0|$ on $U$.

%%%%%%%%%%%%%%%%%%%%%%%%%%%%%%%%%%%%%%%%%%%%%%%%%%%%%%%%% Case 2
\subsection{Proof of the main proposition} 
% \subsection{Proof of Proposition \ref{main lemma}}

By definition, we have the following two lemmas.
% The two lemmas below follow from the definition. % and the setting.
%The following lemma is clear since $d = m_1$.

\begin{lem} \label{lem 1 for main lem: case2} 
It follows that $d \geq j$ for any $j$ such that $b_{ij} \neq 0$.
\end{lem}

More precisely,
$(\gamma, d)$ is maximum in the sense that
$d \geq j$, and $\gamma \geq i$ if $d = j$.

\begin{lem} \label{lem 2 for main lem: case2} 
It follows that
$\gamma + l_1 d \geq i + l_1 j$ and
$\gamma + l_1 d \geq l_1 \delta$
for any $(i,j)$ such that $b_{ij} \neq 0$.
\end{lem}

\begin{proof}
The numbers $l_1 \delta$, $\gamma + l_1 d$ and $i + l_1 j$ are 
the $x$-intercepts of the lines with slope $-l_1^{-1}$
passing through the points $(0, \delta)$, $(\gamma, d)$ and $(i,j)$. 
\end{proof}

Note that
$\gamma + l_1 d = n_2 + l_1 m_2$ and 
$\gamma + l_1 d > n_j + l_1 m_j$ for any $j \geq 3$.
%%%
%These inequalities in Lemmas \ref{lem 1 for main lem: case2} 
%and \ref{lem 2 for main lem: case2}  induce the main lemma. 
Let
\[
\zeta (z) = \frac{p(z) - z^{\delta}}{z^{\delta}}
\text{ and }
\eta (z,w) = \frac{q(z,w) - z^{\gamma} w^d}{ z^{\gamma} w^d}.
\]

\begin{proof}[Proof of Proposition \ref{main lemma} for Case 2]  
We first % define and 
show the former statement.  
It is clear that, 
for any small $\varepsilon$,
there is $R$ such that
$|\zeta| < \varepsilon$ on $U$. 
Let $l=l_1$ and $|w| = |z^l c|$.
Then 
%\[
%U = \{ |z| > R, |w| > R |z|^{l} \} = \{ |z| > R, |c| > R \} \text{ and}
%\] 
\begin{align*}
U
& = \{ |z| > R, |w| > R |z|^{l} \} = \{ |z| > R, |c| > R \} \text{ and} \\
|\eta (z,w)|
& = \left| \sum \frac{b_{ij} z^{i} w^{j}}{z^{\gamma} w^{d}} \right| 
= \left| \sum \frac{b_{ij} z^{i} (z^l c)^{j}}{z^{\gamma} (z^l c)^{d}} \right| 
= \left| \sum \frac{b_{ij} z^{i +lj} c^{j}}{z^{\gamma + ld} c^{d}} \right| 
 \leq \sum \dfrac{|b_{ij}|}{|z|^{(\gamma + ld) - (i + lj)} |c|^{d - j}},
%& \leq \sum |b_{ij}| |z|^{(i + lj) - (\gamma + ld)} |c|^{j - d}.
\end{align*}
where the sum is taken over all $(i,j) \neq (\gamma, d)$
such that $b_{ij} \neq 0$.
It follows from Lemmas \ref{lem 1 for main lem: case2} 
and \ref{lem 2 for main lem: case2} that
$\gamma + ld \geq i + lj$ and $d \geq j$.
Moreover,
for each $(i,j) \neq (\gamma, d)$,
at least one of the inequalities $(\gamma + ld) - (i + lj) > 0$ and $d - j > 0$ holds.
% since $d \geq j$, and $\gamma > i$ if $j = d$.
More precisely,
$(\gamma + ld) - (i + lj) \geq \gamma - i \geq 1$ 
and/or $d - j \geq 1$.
Therefore, 
for any small $\varepsilon$,
there is $R$ such that
$|\eta| < \varepsilon$ on $U$. 

We next show the invariance of $U$.
Since the inequality $|p(z)| > R$ is trivial,
it is enough to show that
$|q(z,w)| > R |p(z)|^{l}$ for any $(z,w)$ in $U$.
% Because $\gamma + ld \geq l \delta$,
We have that
\[
\left| \frac{q(z,w)}{p(z)^{l}} \right| 
\sim \left| \frac{z^{\gamma} w^{d}}{(z^{\delta})^{l}} \right| 
= \left| \frac{z^{\gamma} (z^l c)^{d}}{(z^{\delta})^{l}} \right| 
= |z|^{\gamma + ld - l \delta}  |c|^{d}
\]
on $U$ as $R \to \infty$.
Let $\tilde{\gamma} = \gamma + l d - l \delta$.
Then $\tilde{\gamma} \geq 0$ by
Lemma \ref{lem 2 for main lem: case2}.
If $d \geq 2$, 
then $|z|^{\tilde{\gamma}}  |c|^{d} \geq |c|^d > R^d$ and so
$|q/p^{l}| \geq C R^d > R$ 
% $|q(z,w)/p(z)^{l}| \geq C R^d > R$ 
for some constant $C$ and sufficiently large $R$.
If $d = 1$ and $\delta < T_1$,
then $\tilde{\gamma} > 0$ and so
$|z|^{\tilde{\gamma}}  |c|^{d} > R^{\tilde{\gamma} + 1}$.
Hence $|q/p^{l}| \geq C R^{\tilde{\gamma} + 1} > R$ 
% Hence $|q(z,w)/p(z)^{l}| \geq C R^{\tilde{\gamma} + 1} > R$ 
for some constant $C$ and sufficiently large $R$.
\end{proof} 

%%%%%%%%%%%%%%%%%%%%%%%%%%%%%%%%%%%%%%%%%%%%%%%%%%%%%%%%% Case 2
\subsection{Blow-ups}

Assume that $l_1$ is integer.
% we illustrate our main results in terms of blow-ups.
% Let $p(z) = z^{\delta} \{ 1 + \zeta (z) \}$ and
% $q(z,w) = z^{\gamma} w^d \{ 1 + \eta (z,w) \}$.
Against the previous paper \cite{ueno SA},
we do not assume that $p(z) = z^{\delta}$ here.
Let $\pi_1 (z,c) = (z, z^{l} c)$
and $\tilde{f} = \pi_1^{-1} \circ f \circ \pi_1$,
where $l = l_1$.
Note that $\pi_1$ is the $l$th compositions of the blow-up $(z,c) \to (z,zc)$. 
Then % we have
\begin{align*}
\tilde{f} (z,c) &= (\tilde{p}(z), \tilde{q}(z,c)) = \left( p(z), \ \dfrac{q(z,z^l c)}{p(z)^l} \right)
%\tilde{p}(z) &= z^{\delta} (1 + \zeta (z)) 
\text{ and} \\
\tilde{q}(z,c)
&= \dfrac{ z^{\gamma+ld - l \delta} c^d + \sum b_{ij} z^{i + lj - l \delta} c^{j} }{ \left\{ 1 + \zeta (z) \right\}^l } 
= \dfrac{ z^{\gamma+ld - l \delta} c^d }{ \left\{ 1 + \zeta (z) \right\}^l }
\cdot \left\{ 1 + \sum \dfrac{b_{ij}}{ z^{(\gamma + ld) - (i + lj)} c^{d - j} } \right\}.
\end{align*}
%\begin{align*}
%\tilde{f} (z,c) &= (\tilde{p}(z), \tilde{q}(z,c)), \ 
%\tilde{p}(z) = p(z) = z^{\delta} (1 + \zeta (z)) \text{ and} \\
%\tilde{q}(z,c) &= \dfrac{q(z,z^l c)}{p(z)^l} 
%= \dfrac{ z^{\gamma+ld - l \delta} c^d + \sum b_{ij} z^{i + lj - l \delta} c^{j} }{ \left\{ 1 + \zeta (z) \right\}^l } \\
%&= \dfrac{ z^{\gamma+ld - l \delta} c^d }{ \left\{ 1 + \zeta (z) \right\}^l }
%\cdot \left\{ 1 + \sum \dfrac{b_{ij}}{ z^{(\gamma + ld) - (i + lj)} c^{d - j} } \right\}.
%\end{align*}
%\begin{align*}
%\tilde{f} (z,c) &= (\tilde{p}(z), \tilde{q}(z,c)) = \left( p(z), \ \dfrac{q(z,z^l c)}{p(z)^l} \right) \\
%&= \left( z^{\delta} (1 + \zeta (z)), \ 
%\dfrac{ z^{\gamma+ld - l \delta} c^d + \sum b_{ij} z^{i + lj - l \delta} c^{j} }{ \left\{ 1 + \zeta (z) \right\}^l } \right) \\
%&= \left( z^{\delta} (1 + \zeta (z)), \ 
%\dfrac{ z^{\gamma+ld - l \delta} c^d }{ \left\{ 1 + \zeta (z) \right\}^l }
%\cdot \left\{ 1 + \sum \dfrac{b_{ij}}{ z^{(\gamma + ld) - (i + lj)} c^{d - j} } \right\} \right).
%\end{align*}
Note that $\pi_1^{-1} (U) = \{ |z| > R, |c| > R \}$.

\begin{prop} \label{} 
If $l_1 \in \mathbb{N}$, 
then $\tilde{f}$ is well defined, rational and skew product % and rigid 
on $\mathbb{C}^2$ and holomorphic on $\{ |z| > R \}$. 
More precisely,
\[
\tilde{f} (z,c) = \left( z^{\delta} \{ 1 + \zeta (z) \}, \ 
z^{\gamma + l_1 d - l_1 \delta} c^d \cdot % (1 + \tilde{\eta} (z,c)) \right),
\dfrac{1 + \eta (z,c)}{\left\{ 1 + \zeta (z) \right\}^{l_1}} \right),
\]
where $\zeta$, $\eta \to 0$ 
on $\{ |z| > R, |c| > R \}$ as $R \to \infty$.
\end{prop}

\begin{rem} \label{}  
Even if $l_1$ is rational,
we can lift $f$ to a rational skew product similar to $\tilde{f}$
as stated in Proposition \ref{branched coverings: case2} in Section 5.1.
\end{rem}

As explained below, $\tilde{f}$ is a rational skew product in Case 1.
Therefore,
we can construct the B\"{o}ttcher coordinate for $\tilde{f}$ 
on $\{ |z| > R, |c| > R \}$, % $\pi_1^{-1} (U)$,
which induces the B\"{o}ttcher coordinate for $f$ on $U$.

%We end this subsection with a description of the Newton polygon of the rational function $\tilde{q}$.
We can define the Newton polygon $N(\tilde{q})$ of the rational function $\tilde{q}$ 
in a similar fashion to that of $q$ % as the same as 
by permitting negative indexes and using the Taylor expression of $\zeta$ near infinity.
Let 
$\tilde{\gamma} = \gamma + l_1 d - l_1 \delta$,
$\tilde{i} = i + l_1 j - l_1 \delta$ and 
$\tilde{n}_k = n_k + l_1 m_k - l_1 \delta$.
Then 
$\tilde{q}(z,c) = ( z^{\tilde{\gamma}} c^d + \sum b_{ij} z^{\tilde{i}} c^{j} ) \{ 1 + \zeta (z) \}^{-l_1}$,
%N(\tilde{q}) = N(z^{\tilde{\gamma}} c^d + \sum b_{ij} z^{\tilde{i}} c^{j})
$N(\tilde{q})$ coincides with 
the Newton polygon of $z^{\tilde{\gamma}} c^d + \sum b_{ij} z^{\tilde{i}} c^{j}$
and the candidates of the vertices are $(\tilde{n}_k, m_k)$'s.
Lemma \ref{lem 2 for main lem: case2} is translated into the following.

\begin{lem} \label{} 
It follows that 
$\tilde{\gamma} \geq \tilde{i}$ and 
$\tilde{\gamma} \geq 0$ for any $(i,j)$ such that $b_{ij} \neq 0$.
\end{lem}

Consequently,
$N(\tilde{q})$ % the Newton polygon of $\tilde{q}$ 
has just one vertex $(\tilde{\gamma}, d)$:
$N(\tilde{q}) = D(\tilde{\gamma}, d)$.
In this sense,
we may say that
the rational skew product $\tilde{f}$ belongs to Case 1.

%%%%%%%%%%%%%%%%%%%%%%%%%%%%%%%%%%%%%%%%%%%%%%%%%%%%%%%%%%%%%%%%%%%%%%%%%%%%
%%%%%%%%%%%%%%%%%%%%%%%%%%%%%%%%%%%%%%%%%%%%%%%%%%%%%%%%%%%%%%%%%%%%%%%%%%%%
\section{Main proposition and Blow-ups for Case 3}

We prove Proposition \ref{main lemma} for Case 3 in this section. 
Let $s > 1$,
\[
T_{s-1} \leq \delta, \ 
(\gamma, d) = (n_{s}, m_{s}), \
l_1 = 0 
\text{ and } l_2 = \frac{n_s - n_{s-1}}{m_{s-1} - m_s}.
\] 
Then $\delta > d$ and $\gamma > 0$. % by the setting.
Similar to the previous section,
we first prove Proposition \ref{main lemma} in Section 3.1 and then
illustrate our main results in terms of blow-ups when $l_2^{-1}$ is integer in Section 3.2.
% assume that $a_{\delta} = 1$ and $b_{\gamma d} = 1$ for simplicity.

%%%%%%%%%%%%%%%%%%%%%%%%%%%%%%%%%%%%%%%%%%%%%%%%%%%%%%%%% Case 3
\subsection{Proof of the main proposition}
% \subsection{Proof of Proposition \ref{main lemma}}

By definition, we have the following two lemmas.
% The two lemmas below follow from the definition. % and the setting.
%The following lemma is clear since $\gamma = n_s$.

\begin{lem} \label{lem1 for main lem: case3} 
It follows that
$\gamma \geq i$ for any $i$ such that $b_{ij} \neq 0$.
\end{lem}

More precisely,
$(\gamma, d)$ is maximum in the sense that
$\gamma \geq i$, and $d \geq j$ if $\gamma = i$.

\begin{lem} \label{lem2 for main lem: case3} 
It follows that
$l_2 \delta \geq \gamma + l_2 d \geq i + l_2 j$ % and $l_2^{-1} \gamma + d \leq \delta$
for any $(i,j)$ such that $b_{ij} \neq 0$.
\end{lem}

%\begin{proof}
%The numbers $l_2^{-1} \gamma + d$ and $l_2^{-1} i + j$ are 
%the $y$-intercepts of the lines with slope $-l_2^{-1}$
%passing through the points $(\gamma, d)$ and $(i,j)$.
%In particular, $l_2^{-1} \gamma + d = T_{s-1} \leq \delta$.
%\end{proof}

Note that
$\gamma + l_2 d = n_{s-1} + l_2 m_{s-1}$ and
$\gamma + l_2 d > n_{j} + l_2 m_{j}$ for any $j \leq s-2$.

%These inequalities in Lemmas \ref{lem1 for main lem: case3} 
%and \ref{lem2 for main lem: case3}  induce the main lemma. 

\begin{proof}[Proof of Proposition \ref{main lemma} for Case 3] 
We first % define $\eta (z,w) = \{ q(z,w) - z^{\gamma} w^d \}/ z^{\gamma} w^d$
show the former statement for $q$.
Let $l=l_2$ and $|z| = |tw^{l^{-1}}|$.
Then $U = \{ |z| > R |w|^{l^{-1}}, |w| > R \} = \{ |t| > R, |w| > R \}$ and
\begin{align*}
% U & = \{ |z| > R |w|^{l^{-1}}, |w| > R \} = \{ |t| > R, |w| > R \} \text{ and} \\
|\eta (z,w)|
& = \left| \sum \frac{b_{ij} z^{i} w^{j}}{z^{\gamma} w^{d}} \right| 
= \left| \sum \frac{b_{ij} (tw^{l^{-1}})^{i} w^{j}}{(tw^{l^{-1}})^{\gamma} w^{d}} \right| 
= \left| \sum \frac{b_{ij} t^{i} w^{l^{-1} i +j}}{t^{\gamma} w^{l^{-1} \gamma + d}} \right| 
 \leq \sum \dfrac{ |b_{ij}| }{ |t|^{\gamma - i} |w|^{(l^{-1} \gamma + d) - (l^{-1} i + j)} },
%& \leq \sum |b_{ij}| |t|^{i - \gamma} |w|^{(l^{-1} i + j) - (l^{-1} \gamma + d)}.
\end{align*}
where the sum is taken over all $(i,j) \neq (\gamma, d)$
such that $b_{ij} \neq 0$.
It follows from Lemmas \ref{lem1 for main lem: case3} 
and \ref{lem2 for main lem: case3} that
$\gamma \geq i$ and $l^{-1} \gamma + d \geq l^{-1} i + j$.
Moreover,
for each $(i,j) \neq (\gamma, d)$,
at least one of the inequalities $\gamma > i$ 
and $l^{-1} \gamma + d > l^{-1} i + j$ holds
since $\gamma \geq i$, and $d > j$ if $i = \gamma$.
More precisely,
$\gamma - i \geq 1$ and/or $(l^{-1} \gamma + d) - (l^{-1} i + j) \geq d - j \geq 1$. 
Therefore, 
for any small $\varepsilon$, 
there is $R$ such that
$|\eta| < \varepsilon$ on $U$. 

We next show the invariance of $U$.
Since the inequality $|q(z,w)| > R$ is trivial,
it is enough to show that
$|p(z)| > R |q(z,w)|^{l^{-1}}$ for any $(z,w)$ in $U$.
We have that % Because $\delta \geq l^{-1} \gamma + d$,
\begin{align*}
\left| \frac{p(z)}{q(z,w)^{l^{-1}}} \right| 
& \sim \left| \frac{z^{\delta}}{(z^{\gamma} w^{d})^{l^{-1}}} \right| 
= \left| \frac{(tw^{l^{-1}})^{\delta}}{\{ (tw^{l^{-1}})^{\gamma} w^{d} \}^{l^{-1}}} \right| 
 = |t|^{\delta - l^{-1} \gamma} |w|^{l^{-1} \{ \delta - (l^{-1} \gamma + d) \}}
\geq |t|^d |w|^{l^{-1} \{ \delta - (l^{-1} \gamma + d) \}}
\end{align*}
on $U$ as $R \to \infty$
because $\delta \geq l^{-1} \gamma + d$.
If $d \geq 2$,
then $|t|^d \geq R^d$ and so
$|p/q^{l^{-1}}| \geq C R^d \geq R$ 
for some constant $C$ and sufficiently large $R$.
If $d = 1$ and $\delta > T_{s-1}$, then
$\delta > l^{-1} \gamma + d$ and so
$|t|^d |w|^{l^{-1} \{ \delta - (l^{-1} \gamma + d) \}} > R^{1 + l^{-1} \{ \delta - (l^{-1} \gamma + d) \}}$.
Hence $|p/q^{l^{-1}}| \geq C R^{1 + l^{-1} \{ \delta - (l^{-1} \gamma + d) \}} \geq R$
for some constant $C$ and sufficiently large $R$.
\end{proof} 

%%%%%%%%%%%%%%%%%%%%%%%%%%%%%%%%%%%%%%%%%%%%%%%%%%%%%%%%% Case 3
\subsection{Blow-ups}

Assume that $l_2^{-1}$ is integer.
% we illustrate our main results in terms of blow-ups.
Let $\pi_2 (t,w) = (t w^{l^{-1}}, w)$ and $\tilde{f} = \pi_2^{-1} \circ f \circ \pi_2$,
where $l = l_2$.  
Note that $\pi_2$ is the $l^{-1}$th compositions of the blow-up $(t,w) \to (tw,w)$.
Then % we have that
\begin{align*}
\tilde{f} (t,w) 
&= (\tilde{p} (t,w), \tilde{q} (t,w))
= \left( \dfrac{p(tw^{l^{-1}})}{q(tw^{l^{-1}},w)^{l^{-1}}}, \ q(tw^{l^{-1}} ,w) \right), \\
\tilde{q} (t,w) &= t^{\gamma} w^{l^{-1} \gamma + d} + \sum b_{ij} t^{i} w^{l^{-1} i + j} 
= t^{\gamma} w^{l^{-1} \gamma + d} 
    \left\{ 1 + \sum \dfrac{ b_{ij} }{ t^{\gamma - i} w^{(l^{-1} \gamma + d) - (l^{-1} i + j)} } \right\} \\
&= t^{\gamma} w^{l^{-1} \gamma + d} \{ 1 + \eta (t,w) \}
\text{ and so} \\
\tilde{p} (t,w)
&= t^{\delta - l^{-1} \gamma} w^{l^{-1} \{ \delta - (l^{-1} \gamma + d) \}} \cdot
\dfrac{1 + \zeta (tw^{l^{-1}})}{\{ 1 + \eta (t,w) \}^{l^{-1}}}.
\end{align*}
%\begin{align*}
%\tilde{f} (t,w) &= (\tilde{p} (t,w), \tilde{q} (t,w)), \\
%\tilde{q} (t,w) &= q(tw^{l^{-1}}, w)
%= t^{\gamma} w^{l^{-1} \gamma + d} + \sum b_{ij} t^{i} w^{l^{-1} i + j} \\
%&= t^{\gamma} w^{l^{-1} \gamma + d} 
%    \left\{ 1 + \sum \dfrac{ b_{ij} }{ t^{\gamma - i} w^{(l^{-1} \gamma + d) - (l^{-1} i + j)} } \right\} \\
%&= t^{\gamma} w^{l^{-1} \gamma + d} \{ 1 + \eta (t,w) \} \text{ and so} \\
%\tilde{p} (t,w) &= \dfrac{p(tw^{l^{-1}})}{q(tw^{l^{-1}},w)^{l^{-1}}}
%= t^{\delta - l^{-1} \gamma} w^{l^{-1} \{ \delta - (l^{-1} \gamma + d) \}} \cdot
%\dfrac{1 + \zeta (tw^{l^{-1}})}{\{ 1 + \eta (t,w) \}^{l^{-1}}}.
%\end{align*}
%\[
%\tilde{f} (t,w) = (\tilde{p} (t,w), \tilde{q} (t,w))
%= \left( \dfrac{p(tw^{l^{-1}})}{q(tw^{l^{-1}},w)^{l^{-1}}}, \ q(tw^{l^{-1}} ,w) \right)
%\]
%\[
%= \left( t^{\delta - l^{-1} \gamma} w^{l^{-1} \{ \delta - (l^{-1} \gamma + d) \}} \cdot
%\dfrac{1+u(tw^{l^{-1}})}{\{ 1 + \eta (t,w) \}^{l^{-1}}}, 
%\ t^{\gamma} w^{l^{-1} \gamma + d} \{ 1 + \eta (t,w) \} \right).
%\]
Note that $\pi_2^{-1} (U) = \{ |t| > R, |w| > R \}$.

\begin{prop} \label{} 
If $l_2^{-1} \in \mathbb{N}$, 
then $\tilde{f}$ is well defined and rational on $\mathbb{C}^2$ 
and holomorphic on $\{ |t| > R, |w| > R \}$. 
More precisely,
\[
\tilde{f} (t,w) = 
\left( t^{\delta - l_2^{-1} \gamma} w^{l_2^{-1} \{ \delta - (l_2^{-1} \gamma + d) \}} \{ 1 + \tilde{\zeta} (t,w) \},
\ t^{\gamma} w^{l_2^{-1} \gamma + d} \{ 1 + \eta (t,w) \} \right),
\]
where $\tilde{\zeta}$, $\eta \to 0$ on $\{ |t| > R, |w| > R \}$ as $R \to \infty$.
\end{prop}

Although $\tilde{f}$ is not skew product,
it is a perturbation of a monomial map on $\pi_2^{-1} (U)$.
Hence we can construct the B\"{o}ttcher coordinate for $\tilde{f}$ on $\pi_2^{-1} (U)$
by similar arguments as in Section 7 of this paper,
% or one may refer to \cite[pp.498-499]{f}. 
% This conjugacy
which induces the B\"{o}ttcher coordinate for $f$ on $U$.

%%%%%%%%%%%%%%%%%%%%%%%%%%%%%%%%%%%%
% \subsection{Newton polygons}
% It is helpful to consider 

%We end this subsection with a description of 
%the Newton polygon $N(\tilde{q})$ of the polynomial $\tilde{q}$.
Let 
$\tilde{d} = l_2^{-1} \gamma + d$ and  
$\tilde{j} = l_2^{-1} i + j$.
Then $\tilde{q}(t,w) = t^{\gamma} w^{\tilde{d}} + \sum b_{ij} t^{i} w^{\tilde{j}}$
and Lemma \ref{lem2 for main lem: case3} is translated into the following.

\begin{lem} \label{} 
It follows that
$\tilde{d} \geq \tilde{j}$ for any $(i,j)$ such that $b_{ij} \neq 0$.
\end{lem}

Consequently, % Therefore,
the Newton polygon $N(\tilde{q})$ of $\tilde{q}$ % the polynomial $\tilde{q}$ 
has just one vertex $(\gamma, \tilde{d})$:
$N(\tilde{q}) = D(\gamma, \tilde{d})$.

%%%%%%%%%%%%%%%%%%%%%%%%%%%%%%%%%%%%%%%%%%%%%%%%%%%%%%%%%%%%%%%%%%%%%%%%%%%%
%%%%%%%%%%%%%%%%%%%%%%%%%%%%%%%%%%%%%%%%%%%%%%%%%%%%%%%%%%%%%%%%%%%%%%%%%%%%
\section{Main proposition and Blow-ups for Case 4}

We prove Proposition \ref{main lemma} for Case 4 in this section,
which completes the proof of the proposition.
Let  $s > 2$, $T_{k-1} \leq \delta \leq T_{k}$ for some $2 \leq k \leq s-1$,
\[
(\gamma, d) = (n_k, m_k), \ 
l_1 = \frac{n_{k+1} - n_k}{m_k - m_{k+1}} 
\text{ and } l_1 + l_2 = \frac{n_k - n_{k-1}}{m_{k-1} - m_k}. 
\] 
Then $\delta > d$ and $\gamma > 0$. % by the setting.
Against the previous two sections,
we first illustrate our main results in terms of blow-ups in Section 4.1
and then prove Proposition \ref{main lemma} in Section 4.2.
% The lemma below follows from the definition. % s and the setting.
% We have the same inequalities as in Cases 2 and 3 for $l_1$ and $l_1 + l_2$.
By definition, we have the following lemma.

\begin{lem} \label{inequalities for case 4} 
It follows that
$\gamma + l_1 d \geq i + l_1 j$ and $\gamma + l_1 d \geq l_1 \delta$
and that
$(l_1 + l_2) \delta \geq \gamma + (l_1 + l_2) d \geq i + (l_1 + l_2) j$ 
for any $(i,j)$ such that $b_{ij} \neq 0$.
\end{lem}

Note that
$\gamma + l_1 d = n_{k+1} + l_1 m_{k+1}$ and 
$\gamma + l_1 d > n_j + l_1 m_j$ for any $j \neq k$, $k+1$
and that
$\gamma + (l_1 + l_2) d = n_{k-1} + l_1 m_{k-1}$ and 
$\gamma + (l_1 + l_2) d > n_j + (l_1 + l_2) m_j$ for any $j \neq k-1$, $k$.

%%%%%%%%%%%%%%%%%%%%%%%%%%%%%%%%%%%%%%%%%%%%%%%%%%%%%%%%% Case 4
\subsection{Blow-ups}

Assuming that $l_1$ and $l_2^{-1}$ are integer,
we blow-up $f$ to a nice rational map
for which the B\"{o}ttcher coordinate exists on a region near infinity.
%%%
The strategy is to combine the blow-ups in Cases 2 and 3.
We first blow-up $f$ to $\tilde{f}_1$
by $\pi_1$ as in Case 2.
It then turns out that $\tilde{f}_1$ is a rational skew product in Case 3.
We next blow-up $\tilde{f}_1$ to $\tilde{f}_2$
by $\pi_2$ as in Case 3.
The map $\tilde{f}_2$ is a perturbation of a monomial map on a region near infinity,
and we obtain the B\"{o}ttcher coordinates.

%%%%%%%%%%%%%%%%%%%%%%%%%%%%%%%%%%%%%% Case 4
\subsubsection{First blow-up}

Let $\tilde{\gamma} = \gamma + l_1 d - l_1 \delta$
and $\tilde{i} = i + l_1 j - l_1 \delta$ 
as in Case 2.
Then the former statement of Lemma \ref{inequalities for case 4} is translated into the following.

\begin{lem} \label{inequalities for first blow-up} 
It follows that
$\tilde{\gamma} \geq \tilde{i}$ and
$\tilde{\gamma} \geq 0$ for any $(i,j)$ such that $b_{ij} \neq 0$.
\end{lem}

More precisely,
$(\tilde{\gamma}, d)$ is maximum in the sense that
$\tilde{\gamma} \geq \tilde{i}$, and $d \geq j$ if $\tilde{\gamma} = \tilde{i}$.
Note that
$\tilde{\gamma} = \tilde{n}_{k+1}$ and
$\tilde{\gamma} > \tilde{n}_{j}$ for any $j \neq k$, $k+1$.

Let $\pi_1 (z,c) = (z, z^{l_1} c)$
and $\tilde{f}_1 = \pi_1^{-1} \circ f \circ \pi_1$ as in Case 2.
Then 
%\begin{align*}
%\tilde{f}_1 (z,c) &= (\tilde{p}_1 (z), \tilde{q}_1 (z,c)), \  
%\tilde{p}_1 (z) = p(z) =  z^{\delta} \{ 1 + \zeta (z) \} \text{ and} \\
%\tilde{q}_1 (z,c) &= \dfrac{q(z,z^{l_1} c)}{p(z)^{l_1}}
%= \dfrac{ z^{\tilde{\gamma}} c^d + \sum b_{ij} z^{\tilde{i}} c^{j} }{ \{ 1 + \zeta (z) \}^{l_1} }.
%\end{align*}
%\begin{align*}
%\tilde{f}_1 (z,c) &= (\tilde{p}_1 (z), \tilde{q}_1 (z,c))
%= \left( p(z), \ \dfrac{q(z,z^{l_1} c)}{p(z)^{l_1}} \right)
%\text{ and} \\
%\tilde{q}_1 (z,c)
%&= \dfrac{ z^{\tilde{\gamma}} c^d + \sum b_{ij} z^{\tilde{i}} c^{j} }{ \{ 1 + \zeta (z) \}^{l_1} }.
%\end{align*}
\begin{align*}
\tilde{f}_1 (z,c) &= (\tilde{p}_1 (z), \tilde{q}_1 (z,c))
= \left( p(z), \ \dfrac{q(z,z^{l_1} c)}{p(z)^{l_1}} \right) 
= \left( z^{\delta} \{ 1 + \zeta (z) \}, 
  \dfrac{ z^{\tilde{\gamma}} c^d + \sum b_{ij} z^{\tilde{i}} c^{j} }{ \{ 1 + \zeta (z) \}^{l_1} } \right).
\end{align*}
%&= \left( z^{\delta} \{ 1 + \zeta (z) \}, 
%\dfrac{ z^{\tilde{\gamma}} c^d + \sum b_{ij} z^{\tilde{i}} c^{j} }{ \{ 1 + \zeta (z) \}^{l_1} } \right).
%\end{align*}
Note that
$\pi_1^{-1} (U) = \{ |z| > R |c|^{l_2^{-1}}, |c| > R \}
\subset \{ |z| > R^{1 + l_2^{-1}} \}$. 

\begin{prop} \label{} 
If $l_1 \in \mathbb{N}$, 
then $\tilde{f}_1$ is well defined, rational and skew product on $\mathbb{C}^2$ 
and holomorphic on $\{ |z| > R \}$.  
More precisely,
\[
\tilde{f}_1 (z,c) =
\left( z^{\delta} \{ 1 + \zeta (z) \}, 
\big( z^{\tilde{\gamma}} c^d + \sum b_{ij} z^{\tilde{i}} c^{j} \big)
\{ 1 + \eta_1 (z) \} \right),
\]
where $\zeta$, $\eta_1 \to 0$ as $z \to \infty$.
\end{prop}

Note that 
$(\tilde{\gamma}, d)$ is the vertex of the Newton polygon $N(\tilde{q}_1)$
whose $x$-coordinate is maximum
and that
$N(\tilde{q}_1)$ has other vertices such as $(\tilde{n}_{k-1}, m_{k-1})$.
Hence the situation resembles that of Case 3.

% We illustrate that $\tilde{f}_1$ is actually in Case 3. 
Let us show that $\tilde{f}_1$ is actually in Case 3. 
Recall that $L_{k-1}$ is the line passing through 
the vertices $(\gamma, d)$ and $(n_{k-1}, m_{k-1})$,
and $T_{k-1}$ is the $y$-intercept of $L_{k-1}$.
The slope of $L_{k-1}$ is $-(l_1 + l_2)^{-1}$ 
and so $T_{k-1} = (l_1 + l_2)^{-1} \gamma + d$.
Let $\tilde{L}_{k-1}$ be the line passing through  
the vertices $(\tilde{\gamma}, d)$ and $(\tilde{n}_{k-1}, m_{k-1})$,
and $\tilde{T}_{k-1}$ the $y$-intercept of $\tilde{L}_{k-1}$,
where $\tilde{n}_{k-1} = n_{k-1} + l_1 m_{k-1} - l_1 \delta$.
Then the slope of $\tilde{L}_{k-1}$ is $-l_2^{-1}$ 
and so $\tilde{T}_{k-1} = l_2^{-1} \tilde{\gamma} + d$.
The assumption $T_{k-1} \leq \delta$ implies the following lemma and proposition.

\begin{lem} \label{} 
It follows that
$\tilde{T}_{k-1} \leq \delta$. 
More precisely,
$\tilde{T}_{k-1} < \delta$ if $T_{k-1} < \delta$,
and $\tilde{T}_{k-1} = \delta$ if $T_{k-1} = \delta$.
\end{lem}

\begin{proof}
Since $T_{k-1} = (l_1 + l_2)^{-1} \gamma + d \leq \delta$,
$\gamma + (l_1 + l_2)d \leq (l_1 + l_2) \delta$ 
and so $\gamma + l_1 d - l_1 \delta + l_2 d \leq l_2 \delta$.
Hence $\tilde{T}_{k-1} = l_2^{-1} \tilde{\gamma} + d = l_2^{-1} (\gamma + l_1 d - l_1 \delta) + d \leq \delta$.
\end{proof}

\begin{prop} \label{} 
If $l_1 \in \mathbb{N}$, 
then $\tilde{f}_1$ is a rational skew product in Case 3. 
\end{prop}

%%%%%%%%%%%%%%%%%%%%%%%%%%%%%%%%%%%%%% Case 4
\subsubsection{Second blow-up}

The latter statement of Lemma \ref{inequalities for case 4} is translated into the following:
we have the same inequalities as in Case 3 for $\tilde{\gamma}$ and $\tilde{i}$,
instead for $\gamma$ and $i$.

\begin{lem} \label{} 
It follows that
$l_2 \delta \geq \tilde{\gamma} + l_2 d \geq \tilde{i} + l_2 j$ 
for any $(i,j)$ such that $b_{ij} \neq 0$.
\end{lem}

%\begin{proof}
%The numbers $l_2^{-1} \tilde{\gamma} + d$ and $l_2^{-1} \tilde{i} + j$ are 
%the $y$-intercepts of the lines with slope $-l_2^{-1}$
%passing through the points $(\tilde{\gamma}, d)$ and $(\tilde{i},j)$.
%In particular, $l_2^{-1} \tilde{\gamma} + d = \tilde{T}_{k-1} \leq \delta$.
%\end{proof}

Let $\tilde{d} = l_2^{-1} \tilde{\gamma} + d$
and $\tilde{j} = l_2^{-1} \tilde{i} + j$
as in Case 3.
Then this lemma implies the following.
% is translated into the following.

\begin{lem} \label{inequalities for second blow-up} 
It follows that
$\delta \geq \tilde{d} \geq \tilde{j}$ for any $(i,j)$ such that $b_{ij} \neq 0$.
\end{lem}

Note that
$\tilde{d} = \tilde{m}_{k-1}$ and
$\tilde{d} > \tilde{m}_j$ for any $j \neq k-1$, $k$.
%%%
In particular,
the maximality of $(\tilde{\gamma}, \tilde{d})$ follows from 
Lemmas \ref{inequalities for first blow-up} and \ref{inequalities for second blow-up}.

\begin{cor}\label{cor for second blow-up}
It follows that
$\tilde{\gamma} \geq \tilde{i}$ and $\tilde{d} \geq \tilde{j}$ for any $(i,j)$ such that $b_{ij} \neq 0$.
\end{cor}

Let $\pi_2 (t,c) = (t c^{l_2^{-1}}, c)$ and $\tilde{f}_2 = \pi_2^{-1} \circ \tilde{f}_1 \circ \pi_2$ as in Case 3.
Then 
\begin{align*}
\tilde{f}_2 (t,c) 
&= \left( \tilde{p}_2 (t,c), \tilde{q}_2 (t,c) \right) 
= \left( \tilde{q}_1 (tc^{l_2^{-1}},c), \dfrac{\tilde{p}_1 (tc^{l_2^{-1}})}{\tilde{q}_1 (tc^{l_2^{-1}},c)^{l_2^{-1}}} \right), \\
\tilde{q}_2 (t,c) 
&= \left\{ (tc^{l_2^{-1}})^{\tilde{\gamma}} c^d + \sum b_{ij} (tc^{l_2^{-1}})^{\tilde{i}} c^{j} \right\} \{ 1 + \eta_1 (tc^{l_2^{-1}}) \} 
= \left\{ t^{\tilde{\gamma}} c^{\tilde{d}} + \sum b_{ij} t^{\tilde{i}} c^{\tilde{j}} \right\} \{ 1 + \eta_1 (tc^{l_2^{-1}}) \} \\
&= t^{\tilde{\gamma}} c^{\tilde{d}} \left\{ 1 + \sum \dfrac{b_{ij}}{ t^{\tilde{\gamma} - \tilde{i}} c^{\tilde{d} - \tilde{j}}} \right\} 
\{ 1 + \eta_1 (tc^{l_2^{-1}}) \} 
= t^{\tilde{\gamma}} c^{\tilde{d}} \left\{ 1 + \eta_2 (t,c) \right\}
\text{ and so} \\
\tilde{p}_2 (t,c)
&= t^{\delta - l_2^{-1} \tilde{\gamma}} c^{l_2^{-1} (\delta - \tilde{d})} 
\dfrac{ 1 + \zeta (tc^{l_2^{-1}}) }{ \{ 1 + \eta_2 (t,c) \}^{l_2^{-1}} }.
\end{align*}
%\begin{align*}
%\tilde{f}_2 (t,c) &= \left( \tilde{p}_2 (t,c), \tilde{q}_2 (t,c) \right)
%= \left( \dfrac{\tilde{p}_1 (tc^{l_2^{-1}})}{\tilde{q}_1 (tc^{l_2^{-1}},c)^{l_2^{-1}}}, \ \tilde{q}_1 (tc^{l_2^{-1}},c) \right) \\
%&= \left( t^{\delta - l_2^{-1} \tilde{\gamma}} c^{l_2^{-1} (\delta - \tilde{d})} 
%\dfrac{ 1 + \zeta (tc^{l_2^{-1}}) }{ \{ 1 + \eta_2 (t,c) \}^{l_2^{-1}} },
%\ t^{\tilde{\gamma}} c^{\tilde{d}} \{ 1 + \eta_2 (t,c) \} \right).
%\end{align*}
Note that
$\pi_2^{-1} (\pi_1^{-1} (U)) = \{ |t| > R, |c| > R \}$.

\begin{prop} \label{} 
If $l_1, l_2^{-1} \in \mathbb{N}$, 
then $\tilde{f}_2$ is well defined and rational on $\mathbb{C}^2$ % and rigid
and holomorphic on $\{ |t| > R, |c| > R \}$. 
More precisely,
\[
\tilde{f}_2 (t,c) =
\left( t^{\delta - l_2^{-1} \tilde{\gamma}} c^{l_2^{-1} (\delta - \tilde{d})} \{ 1 + \zeta_2 (t,c) \}, 
\ t^{\tilde{\gamma}} c^{\tilde{d}} \{ 1 + \eta_2 (t,c) \} \right),
\]
where $\zeta_2$, $\eta_2 \to 0$ on $\{ |t| > R, |c| > R \}$ as $R \to \infty$.
\end{prop}

Therefore,
we can construct the B\"{o}ttcher coordinate for $\tilde{f}_2$
on $\pi_2^{-1} (\pi_1^{-1} (U))$, % $\{ |t| > R, |c| > R \}$,
which induces that for $\tilde{f}_1$ 
on $\pi_1^{-1} (U)$ and % $\{ |z| > R |c|^{l_2^{-1}}, |c| > R \}$ and 
that for $f$ 
on $U$.  

%Although the Newton polygon of $\tilde{q}_1$ has at least two vertices,
As the same as the previous subsections,
the Newton polygon $N(\tilde{q}_2)$ of the rational function $\tilde{q}_2$ 
has just one vertex $(\tilde{\gamma}, \tilde{d})$:
$N(\tilde{q}_2) = D(\tilde{\gamma}, \tilde{d})$.

%%%%%%%%%%%%%%%%%%%%%%%%%%%%%%%%%%%%%%%%%%%%%%%%%%%%%%%%% Case 4
\subsection{Proof of the main proposition}
% \subsection{Proof of Proposition \ref{main lemma}}

The idea of the blow-ups in the previous subsection
provides a proof of Proposition \ref{main lemma}.
Because we take the absolute value in the proof,
we do not need to care whether $\tilde{f}_1$ and $\tilde{f}_2$ are well defined.

\begin{proof}[Proof of Proposition \ref{main lemma} for Case 4] 
We first % define $\eta (z,w) = \{ q(z,w) - z^{\gamma} w^d \}/ z^{\gamma} w^d$
show the former statement for $q$.
Let $|w| = |z^{l_1}c|$ and $|z| = |tc^{l_2^{-1}}|$.
Then $U = \{ |z|^{l_1 + l_2} > R^{l_2} |w|, |w| > R |z|^{l_1} \} = \{ |t| > R, |c| > R \}$,
%\begin{align*}
%U &= \{ |z|^{l_1 + l_2} > R^{l_2} |w|, |w| > R |z|^{l_1} \} = \{ |t| > R, |c| > R \}, \\
%\end{align*}
\begin{align*}
\left| \frac{z^{i} w^{j}}{z^{\gamma} w^{d}} \right| 
&= \left| \frac{z^{i} (z^{l_1}c)^{j}}{z^{\gamma} (z^{l_1}c)^{d}} \right| 
= \left| \frac{z^{i +l_1 j} c^{j}}{z^{\gamma + l_1 d}c^{d}} \right|
= \left| \frac{z^{\tilde{i}} c^{j}}{z^{\tilde{\gamma}}c^{d}} \right| 
= \left| \frac{(tc^{l_2^{-1}})^{\tilde{i}} c^{j}}{(tc^{l_2^{-1}})^{\tilde{\gamma}} c^{d}} \right|
= \left| \frac{t^{\tilde{i}} c^{l_2^{-1} \tilde{i} + j}}{t^{\tilde{\gamma}} c^{l_2^{-1} \tilde{\gamma} + d}} \right|
= \left| \frac{t^{\tilde{i}} c^{\tilde{j}}}{t^{\tilde{\gamma}} c^{\tilde{d}}} \right|
\text{ and so} \\
|\eta (z,w)|
& \leq \sum \dfrac{|b_{ij}|}{ |t|^{\tilde{\gamma} - \tilde{i}} |c|^{\tilde{d} - \tilde{j}} },
\end{align*}
where the sum is taken over all $(i,j) \neq (\gamma, d)$
such that $b_{ij} \neq 0$.
It follows from Corollary \ref{cor for second blow-up} that
$\tilde{\gamma} \geq \tilde{i}$ and $\tilde{d} \geq \tilde{j}$.
Moreover, 
for each $(i,j) \neq (\gamma, d)$,
at least one of the inequalities 
$\tilde{\gamma} - \tilde{i} > 0$ and 
$\tilde{d} - \tilde{j} > 0$ holds. 
More precisely,
$\tilde{\gamma} - \tilde{i} > 0$ and/or 
$\tilde{d} - \tilde{j} = j - d + l_2^{-1} (\tilde{i} - \tilde{\gamma}) \geq 1$.
Therefore, 
for any small $\varepsilon$,
there is $R$ such that
$|\eta| < \varepsilon$ on $U$. 

We next show the invariance of $U$.
More precisely,
we show that 
$|p(z)^{1 + l_1 l_2^{-1}}| > R |q(z,w)^{l_2^{-1}}|$ and
$|q(z,w)| > R |p(z)^{l_1}|$
for any $(z,w)$ in $U$.
Note that $|z| = |tc^{l_2^{-1}}|$ and $|w| = |t^{l_1} c^{1 + l_1 l_2^{-1}}|$.
%and that, formally,
%\[
%\tilde{f}_2 (t,c) =
%\left( \dfrac{\tilde{p}_1 (tc^{l_2^{-1}})}{\tilde{q}_1 (tc^{l_2^{-1}},c)^{l_2^{-1}}}, \ \tilde{q}_1 (tc^{l_2^{-1}},c) \right)
%\]
%\[
%= \left( \dfrac{p(tc^{l_2^{-1}})^{1+l_1 l_2^{-1}}}{q(tc^{l_2^{-1}}, t^{l_1} c^{1+l_1 l_2^{-1}})^{l_2^{-1}}}, 
%\dfrac{q(tc^{l_2^{-1}}, t^{l_1} c^{1+l_1 l_2^{-1}})}{p(tc^{l_2^{-1}})^{l_1}} \right).
%\]
Because $\delta \geq \tilde{d} = l_2^{-1} \tilde{\gamma} + d$,
\begin{align*}
\left| \frac{p(z)^{1 + l_1 l_2^{-1}}}{q(z,w)^{l_2^{-1}}} \right| 
&\sim \left| \frac{(z^{\delta})^{1 + l_1 l_2^{-1}} }{(z^{\gamma} w^{d})^{l_2^{-1}}} \right| 
= \left| \frac{ \{ (tc^{l_2^{-1}})^{\delta} \}^{1 + l_1 l_2^{-1}} }{\{ (tc^{l_2^{-1}})^{\gamma} (t^{l_1} c^{1 + l_1 l_2^{-1}})^{d} \}^{l_2^{-1}}} \right| 
= |t|^{\delta - l_2^{-1} \tilde{\gamma}} |c|^{l_2^{-1} (\delta - \tilde{d})}
\geq |t|^d |c|^{l_2^{-1} (\delta - \tilde{d})}
\end{align*}
on $U$ as $R \to \infty$.
If $d \geq 2$,
then 
$|t|^{d} |c|^{l_2^{-1} (\delta - \tilde{d})} \geq |t|^d > R^d$.
If $d = 1$ and $\delta > T_{k-1}$, 
then 
$\delta > \tilde{d}$ and so
$|t|^{d} |c|^{l_2^{-1} (\delta - \tilde{d})}
> R^{1 + l_2^{-1} (\delta - \tilde{d})}$.
Because $\tilde{d} \geq d$,
\[
\left| \frac{q(z,w)}{p(z)^{l_1}} \right| 
\sim \left| \frac{z^{\gamma} w^{d}}{(z^{\delta})^{l_1}} \right| 
= \left| \frac{(tc^{l_2^{-1}})^{\gamma} (t^{l_1} c^{1 + l_1 l_2^{-1}})^{d}}{ \{ (tc^{l_2^{-1}})^{\delta} \}^{l_1}} \right| 
= |t|^{\tilde{\gamma}} |c|^{\tilde{d}}
\geq |t|^{\tilde{\gamma}} |c|^d
\]
on $U$ as $R \to \infty$.
If $d \geq 2$,
then 
$|t|^{\tilde{\gamma}} |c|^d \geq |c|^d > R^d$
since $\tilde{\gamma} \geq 0$.
If $d = 1$ and $\delta < T_{k}$, 
then 
$\tilde{\gamma} > 0$ and so
$|t|^{\tilde{\gamma}} |c|^d > R^{\tilde{\gamma} + 1}$.
Hence we obtain the required inequalities.
\end{proof}

%%%%%%%%%%%%%%%%%%%%%%%%%%%%%%%%%%%%%%%%%%%%%%%%%%%%%%%%%%%%%%%%%%%%%%%%%%%%
%%%%%%%%%%%%%%%%%%%%%%%%%%%%%%%%%%%%%%%%%%%%%%%%%%%%%%%%%%%%%%%%%%%%%%%%%%%%
\section{Intervals of weights and branched coverings}

The rational numbers $l_1$ and $l_2$ are called weights 
in the previous papers \cite{ueno poly} and \cite{ueno SA}.
In this section 
we introduce intervals of weights for each of which the main results hold.
% Proposition \ref{main lemma} holds.
Moreover,
we associate rational weights in the intervals to formal branched coverings of $f$.
These coverings are a generalization of the blow-ups of $f$ in the previous sections.
%and they might be well-defined on suitable regions. % even if the weight is rational.  
We deal with Cases 2, 3 and 4 in Sections 5.1, 5.2 and 5.3, respectively.
For Case 2,
the covering is well defined on a region % branched 
for any rational number in the interval.
%see Proposition \ref{branched coverings: case2}.  
On the other hand,
for Cases 3 and 4,
the case when the covering is well defined on a region seems to be limited, % branched
in which the weight $\alpha_0 = \gamma /(\delta - d)$ appears.
% which is realized at least for the weight $\alpha_0 = \gamma /(\delta - d)$.
% see Corollaries \ref{branched coverings: case3} and \ref{branched coverings: case4}, respectively. 

%%%%%%%%%%%%%%%%%%%%%%%%%%%%%%%%%%%%%%%%%%%%%%%%%%%%%%%%% Case 2
\subsection{Intervals and coverings for Case 2}

In the proof of Proposition \ref{main lemma} for Case 2,
the inequalities $\gamma + l_1 d \geq i + l_1 j$ and
$\gamma + l_1 d \geq l_1 \delta$ played a central role.
We define the interval $\mathcal{I}_f$ as 
\[
\mathcal{I}_f = 
\left\{ \ l > 0 \ | 
\begin{array}{lcr}
\gamma + ld \geq i + l j
\text{ and }
\gamma + ld \geq l \delta 
\text{ for any $i$ and $j$ s.t. } b_{ij} \neq 0
\end{array}
\right\}.
\]
It follows that $\min \mathcal{I}_f = l_1$.
In fact, 
if $\delta > d$, then $\gamma > 0$ and
\begin{align*}
\mathcal{I}_f 
&=
\left[
\max_{j < d} \left\{ \dfrac{i - \gamma}{d - j} \right\},
\dfrac{\gamma}{\delta - d}
\right]
=
\left[
\max_{1 < k \leq s} \left\{ \dfrac{n_k - \gamma}{d - m_k} \right\},
\dfrac{\gamma}{\delta - d}
\right] 
=
\left[
\frac{n_2 - \gamma}{d - m_2},
\dfrac{\gamma}{\delta - d}
\right]
=
\left[
l_1,
\alpha_0
\right],
\end{align*}
which is mapped to $[\delta, T_{1}]$
by the transformation $l \to l^{-1} \gamma + d$.
If $\delta \leq d$, then
the inequality $\gamma + ld \geq l \delta$ is trivial and so
$\mathcal{I}_f =[ l_1, \infty )$.

Let $U^l = \{ |z| > R, |w| > R |z|^{l} \}$.

\begin{prop} 
Proposition \ref{main lemma} and Theorem \ref{main theorem}
in Case 2 hold on $U^l$ for any $l$ in $\mathcal{I}_f$.
\end{prop}

\begin{rem} \label{} 
It follows
that $U^{l_1}$ is the largest region among $U^l$ for any $l$ in $\mathcal{I}_f$ and
that $\mathcal{I}_f \neq \emptyset$ if and only if $\delta \leq T_{1}$.
\end{rem}

%%%%%%%%%%%%%%%%%%%%%%%%%%%%%%%%%%%%%%%% branched coverings for Case 2

Let $l = s/r \in \mathcal{I}_f$, %  \cap \mathbb{Q}
where $s$ and $r$ are coprime positive integers, 
$\pi_1 (\mathsf{z},c) = (\mathsf{z}^r, \mathsf{z}^s c)$
and $\tilde{f} = \pi_1^{-1} \circ f \circ \pi_1$.
Then
$\pi_1$ is formally the composition of $(\mathsf{z},c) \to (\mathsf{z}^r,c)$ and $(z,c) \to (z,z^{s/r} c)$,
% and $\tilde{f}$ is well-defined:
\begin{align*}
\tilde{f} (\mathsf{z},c) 
&= (\tilde{p} (\mathsf{z}), \tilde{q} (\mathsf{z},c))
= \left( p(\mathsf{z}^r)^{1/r}, \dfrac{q(\mathsf{z}^r,\mathsf{z}^s c)}{p(\mathsf{z})^{s/r}}  \right), \\
\tilde{p} (\mathsf{z}) 
&= \mathsf{z}^{\delta} \{ 1 + \zeta (\mathsf{z}^r) \}^{1/r} 
\text{ and} \\
\tilde{q} (\mathsf{z},c) 
&= \dfrac{\mathsf{z}^{r \gamma + sd - s \delta} c^d}{\{ 1 + \zeta (\mathsf{z}^r) \}^{s/r}} 
\cdot \left\{ 1 + \sum \dfrac{b_{ij}}{ z^{(r \gamma + sd) - (ri + sj)} c^{d - j} } \right\}.
\end{align*}
%\begin{align*}
%\tilde{f} (\mathsf{z},c) &= (\tilde{p} (\mathsf{z}), \tilde{q} (\mathsf{z},c)) 
%= \left( p(\mathsf{z}^r)^{1/r}, \ \dfrac{q(\mathsf{z}^r,\mathsf{z}^s c)}{p(\mathsf{z})^{s/r}} \right), \\
%\tilde{p} (\mathsf{z}) &= [ (\mathsf{z}^r)^\delta \{ 1 + \zeta (\mathsf{z}^r) \} ]^{1/r} 
%= \mathsf{z}^{\delta} \{ 1 + \zeta (\mathsf{z}^r) \}^{1/r} \text{ and} \\
%\tilde{q} (\mathsf{z},c) 
%&= \dfrac{\mathsf{z}^{r \gamma + sd - s \delta} c^d}{\{ 1 + \zeta (\mathsf{z}^r) \}^{s/r}} 
%\cdot \left\{ 1 + \sum \dfrac{b_{ij}}{ z^{(r \gamma + sd) - (ri + sj)} c^{d - j} } \right\}
%\end{align*}
Note that $\pi_1^{-1} (U) = \{ |\mathsf{z}| > R^{1/r}, |c| > R \}$.
%Recall that $r \gamma + sd \geq s \delta$,
%$r \gamma + sd \geq ri + sj$ and $d \geq j$.

\begin{prop} \label{branched coverings: case2} 
For any rational number $s/r$ in $\mathcal{I}_f$, 
the lift $\tilde{f}$ is well defined, holomorphic and skew product on $\{ |\mathsf{z}| > R^{1/r} \}$.
% $\pi_1^{-1} (U)$. 
% the preimage of the domain of $f$ by $\pi_1$. 
More precisely,
\[
\tilde{f} (\mathsf{z}, c) = \left( \mathsf{z}^{\delta} \{ 1 + \zeta (\mathsf{z}^r) \}^{1/r} , \ 
\mathsf{z}^{r \gamma + sd - s \delta} c^d \cdot
\dfrac{ 1 + \eta (\mathsf{z},c) }{\{ 1 + \zeta (\mathsf{z}^r) \}^{s/r}} \right),
\]
where $\zeta$, $\eta \to 0$ % as $\mathsf{z}$, $c \to \infty$.
on $ \{ |\mathsf{z}| > R^{1/r}, |c| > R \}$ as $R \to \infty$.
\end{prop}

\begin{rem}[Larger invariant regions for Case 1] \label{larger regions for case 1}
Let 
\[
l_1^* = \inf % \left\{ \, l \in \mathbb{Q} \, | \,
\left\{ \ l \in \mathbb{Q} \ | 
\begin{array}{lcr}
\gamma + ld \geq i + l j
\text{ and }
\gamma + ld \geq l \delta 
\text{ for any $i$ and $j$ s.t. } b_{ij} \neq 0 % \, \right\}.
\end{array}
\right\}. 
\]
For Case 1,
$l_1^* \leq 0$ if it exists;
it always exists if $\delta \leq d$.
For Case 2,
$l_1^* = l_1 > 0$. 
% $l_1^* = \min \mathcal{I}_f = l_1 > 0$. 
It was proved in \cite{ueno poly} that
Proposition \ref{main lemma} and
Theorem \ref{main theorem} hold on
$\{ |z| > R, |w| > R |z|^{l_1^*} \}$
if $l_1^*$ exists.
\end{rem}

%%%%%%%%%%%%%%%%%%%%%%%%%%%%%%%%%%%%%%%%%%%%%%%%%%%%%%%%% Case 3
\subsection{Intervals and coverings for Case 3}

In the proof of Proposition \ref{main lemma} for Case 3,
the inequalities $l_2 \delta \geq \gamma + l_2 d \geq i + l_2 j$ played a central role.
We define the interval $\mathcal{I}_f$ as
\[
\mathcal{I}_f = 
\left\{ \ l > 0 \ | 
\begin{array}{lcr}
l \delta \geq \gamma + ld \geq i + lj
\text{ for any $i$ and $j$ s.t. } b_{ij} \neq 0
\end{array}
\right\}. 
\]
It follows that $\max \mathcal{I}_f = l_2$.
In fact,
since $\delta > d$ and $\gamma > 0$, 
\begin{align*}
\mathcal{I}_f 
&= 
\left[
\dfrac{\gamma}{\delta - d},
\min_{j > d} \left\{ \dfrac{\gamma - i}{j - d}  \right\}
\right]
=
\left[
\dfrac{\gamma}{\delta - d},
\min_{1 \leq k \leq s-1} \left\{ \dfrac{\gamma - n_k}{m_k - d}  \right\}
\right] 
=
\left[
\dfrac{\gamma}{\delta - d},
\frac{\gamma - n_{s-1}}{m_{s-1} - d}
\right]
=
\left[
\alpha_0, % \dfrac{\gamma}{\delta - d},
l_2
\right],
\end{align*}
which is mapped to $[T_1, \delta]$
by the transformation $l \to l^{-1} \gamma + d$.

Let $U^l = \{ |z|^{l} > R^{l} |w|, |w| > R \}$.

\begin{prop} \label{} 
Proposition \ref{main lemma} and Theorem \ref{main theorem}
in Case 3 hold on $U^l$ for any $l$ in $\mathcal{I}_f$.
\end{prop}

\begin{rem} \label{} 
It follows
that $U^{l_2}$ is the largest region among $U^l$ for any $l$ in $\mathcal{I}_f$ and
that $\mathcal{I}_f \neq \emptyset$ if and only if $T_{s-1} \leq \delta$.
\end{rem}

%%%%%%%%%%%%%%%%%%%%%%%%%%%%%%%%%%%%%%%% branched coverings for Case 3

Let $l = s/r \in \mathcal{I}_f$,
where $s$ and $r$ are coprime positive integers, 
$\pi_2 (t, \mathsf{w}) = (t \mathsf{w}^r, \mathsf{w}^s)$ 
and $\tilde{f} = \pi_2^{-1} \circ f \circ \pi_2$.
Then
$\pi_2$ is formally the composition of $(t,\mathsf{w}) \to (t,\mathsf{w}^s)$ and $(t,w) \to (tw^{r/s},w)$, and
\[
\tilde{f} (t, \mathsf{w}) =
\left( \dfrac{p(t \mathsf{w}^r)}{q(t \mathsf{w}^r, \mathsf{w}^s)^{r/s}}, \ q(t \mathsf{w}^r, \mathsf{w}^s)^{1/s} \right).
\]
Since $q(z,w) \sim z^{\gamma} w^d$ on $U^l$ as $R \to \infty$,
it follows formally that 
\[
q(t \mathsf{w}^r, \mathsf{w}^s)^{1/s} 
\sim \{ (t \mathsf{w}^r)^{\gamma} (\mathsf{w}^s)^d \}^{1/s} = (t \mathsf{w}^r)^{\gamma /s} \mathsf{w}^d
% \text{ on } \pi_2^{-1} (U^l).
\]
on $\pi_2^{-1} (U^l) = \{ |t| > R, |\mathsf{w}| > R^{1/s} \}$ as $R \to \infty$.
Hence
$\tilde{f}$ is well defined if $\gamma /s$ is integer.

\begin{prop} \label{} 
If $s/r \in \mathcal{I}_f$ and $\gamma /s \in \mathbb{N}$, 
then $\tilde{f}$ is well defined and holomorphic on $\{ |t| > R, |\mathsf{w}| > R^{1/s} \}$.
More precisely,
\[
\tilde{f} (t, \mathsf{w}) = 
\left( t^{\delta - r \gamma/ s} \mathsf{w}^{r \delta - r (r \gamma/ s + d)} \{ 1 + \tilde{\zeta} (t, \mathsf{w}) \},
\ t^{\gamma/ s} \mathsf{w}^{r \gamma/ s + d} \{ 1 + \eta (t, \mathsf{w}) \} \right),
\]
where $\tilde{\zeta}$, $\eta \to 0$ on $\{ |t| > R, |\mathsf{w}| > R^{1/s} \}$ as $R \to \infty$.
\end{prop}

\begin{cor}
If $s/r = \alpha_0$, then $\tilde{f}$ is well defined on the region above.
\end{cor}

%%%%%%%%%%%%%%%%%%%%%%%%%%%%%%%%%%%%%%%%%%%%%%%%%%%%%%%%% Case 4
\subsection{Intervals and coverings for Case 4}

We define the interval $\mathcal{I}_f^1$ as
\[
\mathcal{I}_f^1 = 
\left\{ \ l_{(1)} > 0 \ \Bigg| 
\begin{array}{lcr}
\gamma + l_{(1)} d \geq n_{j} + l_{(1)} m_{j} \text{ for } j \leq k-1 \\ 
\gamma + l_{(1)} d  >   n_{j} + l_{(1)} m_{j} \text{ for } j \geq k+1 \\ 
\gamma + l_{(1)} d \geq l_{(1)} \delta
\end{array}
\right\},
\]
the interval $\mathcal{I}_f^2$ associated with $l_{(1)}$ in $\mathcal{I}_f^1$ as
\[
\mathcal{I}_f^2  = \mathcal{I}_f^2 (l_{(1)}) = 
\left\{ \ l_{(2)} > 0 \ \Big| 
\begin{array}{lcr}
l_{(2)} \delta \geq \tilde{\gamma} + l_{(2)} d \geq \tilde{i} + l_{(2)} j \\
\text{for any $i$ and $j$ s.t. } b_{ij} \neq 0 %%% スペースを消した
\end{array}
\right\},
\]
where 
$\tilde{\gamma} = \gamma + l_{(1)} d - l_{(1)} \delta$
and $\tilde{i} = i + l_{(1)} j - l_{(1)} \delta$,
and the rectangle $\mathcal{I}_f$ as
\[
\mathcal{I}_f =
% \{ (l_{(1)}, l_{(1)} + l_{(2)}) \ | \ l_{(1)} \in \mathcal{I}_f^1, l_{(2)} \in \mathcal{I}_f^2 \}.
\left\{ \ (l_{(1)}, l_{(1)} + l_{(2)}) \ |
\begin{array}{lcr}
l_{(1)} \in \mathcal{I}_f^1, l_{(2)} \in \mathcal{I}_f^2
\end{array}
\right\}.
\]

Let us calculate the intervals and rectangle more practically.
%Let 
%\[
%\alpha_0 = \frac{\gamma}{\delta - d}.
%\]
%Then $\alpha_0 > 0$ since $\delta > d$ and $\gamma > 0$ by the setting.
Since $n_j < \gamma$ and $m_j > d$ for any $j \leq k-1$,
and $n_j > \gamma$ and $m_j < d$ for any $j \geq k+1$,  
\begin{align*}
\mathcal{I}_f^1 
&= 
\left[
\max_{j \geq k+1}
\left\{ 
\dfrac{n_j - \gamma}{d - m_j} 
\right\},
\min_{j \leq k-1}
\left\{ 
\dfrac{\gamma - n_j}{m_j - d} 
\right\}
\right)
\cap
\left(
0,
\dfrac{\gamma}{\delta - d}
\right] \\
&= 
\left[
\dfrac{n_{k+1} - \gamma}{d - m_{k+1}}, 
\dfrac{\gamma - n_{k-1}}{m_{k-1} - d}
\right)
\cap
\left(
0,
\dfrac{\gamma}{\delta - d}
\right]
= [ l_1, l_1 + l_2 ) \cap ( 0, \alpha_0 ].
\end{align*}
In particular,
$\min \mathcal{I}_f^1 = l_1$. % and, as a remark,
%\[
%\mathcal{I}_f^1 = 
%\left\{ \ l_{(1)} > 0 \ \Bigg| 
%\begin{array}{lcr}
%\gamma + l_{(1)} d \geq n_{k-1} + l_{(1)} m_{k-1} \\
%\gamma + l_{(1)} d > n_{k+1} + l_{(1)} m_{k+1} \\
%\gamma + l_{(1)} d \geq l_{(1)} \delta
%\end{array}
%\right\}.
%\]
On the other hand,
\begin{align*}
\mathcal{I}_f^2 
&=
\left[
\dfrac{\tilde{\gamma}}{\delta - d},
\dfrac{\tilde{\gamma} - \tilde{n}_{k-1}}{m_{k-1} - d}  
\right]
\cap \mathbb{R}_{>0}
=
\left[
\dfrac{\gamma}{\delta - d} - l_{(1)},
\dfrac{\gamma - n_{k-1}}{m_{k-1} - d} - l_{(1)}  
\right] 
\cap \mathbb{R}_{>0} \\
&= [ \alpha_0 - l_{(1)}, l_1 + l_2 - l_{(1)} ]
\cap \mathbb{R}_{>0}.
\end{align*}
If $T_{k-1} < \delta = T_{k}$, 
then it follows from the inequality $l_1 = \alpha_0 < l_1 + l_2$ that
\[
\mathcal{I}_f^1 = \{ \alpha_0 \}, \ 
% \mathcal{I}_f^1 = \{ l_1 \}, \ 
\mathcal{I}_f^2 = (0, l_2]
\text{ and so }
\mathcal{I}_f 
= \{ \alpha_0 \} \times ( \alpha_0, l_1 + l_2 ].
% = \{ l_1 \} \times ( l_1, l_1 + l_2 ].
\]
If $T_{k-1} < \delta < T_{k}$, 
then it follows from the inequality $l_1 < \alpha_0 < l_1 + l_2$ that
\begin{align*}
\mathcal{I}_f^1 &= [ l_1, \alpha_0 ], \ 
\mathcal{I}_f^2 = [ \alpha_0 - l_{(1)}, l_1 + l_2 - l_{(1)} ] \cap \mathbb{R}_{>0} 
\text{ and so } \\
\mathcal{I}_f 
&= [ l_1, \alpha_0 ] \times [ \alpha_0, l_1 + l_2 ] - \{ (\alpha_0, \alpha_0) \}.
\end{align*}
If $T_{k-1} = \delta < T_{k}$, then it follows from the inequality $l_1 < \alpha_0 = l_1 + l_2$ that
\[
\mathcal{I}_f^1 = [ l_1, \alpha_0 ), \ 
% \mathcal{I}_f^1 = [ l_1, l_1 + l_2 ), \ 
\mathcal{I}_f^2 = \{ \alpha_0 - l_{(1)} \}
% \mathcal{I}_f^2 = \{ l_1 + l_2 - l_{(1)} \}
\text{ and so }
\mathcal{I}_f 
= [ l_1, \alpha_0 ) \times \{ \alpha_0 \}.
% = [ l_1, l_1 + l_2 ) \times \{ l_1 + l_2 \}.
\]
In particular, 
% $\min \mathcal{I}_f^1 = l_1$ and
$\max \mathcal{I}_f^2 (l_{1}) = l_2$ and
$\max \{ l_{(1)} + l_{(2)} \ | \ l_{(1)} \in \mathcal{I}_f^1, l_{(2)} \in \mathcal{I}_f^2 \} = l_1 + l_2$.

Let $U^{l_{(1)},l_{(2)}} = \{ |z|^{l_{(1)} + l_{(2)}} > R^{l_{(2)}} |w|, |w| > R |z|^{l_{(1)}} \}$.

\begin{prop} \label{} 
Proposition \ref{main lemma} and Theorem \ref{main theorem}
in Case 4 hold on $U^{l_{(1)},l_{(2)}}$
for any $l_{(1)}$ in $\mathcal{I}_f^1$ and $l_{(2)}$ in $\mathcal{I}_f^2$.
\end{prop}

\begin{rem} \label{} 
It follows that 
$U^{l_1, l_2}$ is the largest region among $U^{l_{(1)},l_{(2)}}$ 
for any $l_{(1)}$ in $\mathcal{I}_f^1$ and $l_{(2)}$ in $\mathcal{I}_f^2$
and that
$\mathcal{I}_f^1 \neq \emptyset$ and $\mathcal{I}_f^2 \neq \emptyset$
if and only if $T_{k-1} \leq \delta \leq T_{k}$.
More precisely,
$\mathcal{I}_f^1 = \emptyset$ if $T_{k} < \delta$, and
$\mathcal{I}_f^2 = \emptyset$ if $\delta < T_{k-1}$.
\end{rem}

%%%%%%%%%%%%%%%%%%%%%%%%%%%%%%%%%%%%%%%% first branched coverings for Case 4 

Let $l_{(1)} = s_1/r_1$,
where $s_1$ and $r_1$ are coprime positive integers, 
$\pi_1 (\mathsf{z},c) = (\mathsf{z}^{r_1}, \mathsf{z}^{s_1} c)$
and $\tilde{f}_1 = \pi_1^{-1} \circ f \circ \pi_1$.
Let $\tilde{\gamma} = r_1 \gamma + s_1 d - s_1 \delta$
and $\tilde{i} = r_1 i + s_1 j - s_1 \delta$.
Then 
\begin{align*}
\tilde{f}_1 (\mathsf{z},c) 
&= (\tilde{p}_1 (\mathsf{z}), \tilde{q}_1 (\mathsf{z},c))
= \left( p(\mathsf{z}^{r_1})^{1/r_1}, 
   \dfrac{q(\mathsf{z}^{r_1}, \mathsf{z}^{s_1} c)}{p(\mathsf{z})^{s_1/r_1}} \right) \\
&= \left( \mathsf{z}^{\delta} \left\{ 1 + \zeta (\mathsf{z}^{r_1}) \right\}^{1/r_1},
   \dfrac{ \mathsf{z}^{\tilde{\gamma}} c^d + \sum b_{ij} \mathsf{z}^{\tilde{i}} c^{j} }
        { \left\{ 1 + \zeta (\mathsf{z}^{r_1}) \right\}^{s_1/r_1} }  \right).
\end{align*}
Note that 
$\pi_1^{-1} (U^{l_{(1)},l_{(2)}}) 
= \{ | \mathsf{z} |^{r_1 l_{(2)}} > R^{l_{(2)}} |c|, |c| > R \}
\subset \{ |\mathsf{z}| > R^{(1+l_{(2)}^{-1})/r_1} \}$.

\begin{prop} \label{} 
For any rational number $s_1/r_1$ in $\mathcal{I}_f^1$, 
the lift $\tilde{f}_1$ is well defined, holomorphic and skew product on $\{ |\mathsf{z}| > R^{1/r_1} \}$.
More precisely,
\[
\tilde{f}_1 (\mathsf{z},c) = \left( \mathsf{z}^{\delta} \left\{ 1 + \zeta_1 (\mathsf{z}) \right\}, \ 
\mathsf{z}^{\tilde{\gamma}} c^d \left\{ 1 + \eta_1 (\mathsf{z}, c) \right\}  \right),
\]
where $\zeta_1$, $\eta_1 \to 0$ 
on $\{ | \mathsf{z} |^{r_1 l_{(2)}} > R^{l_{(2)}} |c|, |c| > R \}$ % $\{ R < |c| < R^{- l_2} | \mathsf{z} |^{r_1 l_2} \}$ 
as $R \to \infty$. 
\end{prop}

\begin{rem}
If we defined the interval $\mathcal{I}_f^1$ as
\[
\left\{ \ l_{(1)} > 0 \ | 
\begin{array}{lcr}
\gamma + l_{(1)} d \geq i + l_{(1)} j
\text{ and }
\gamma + l_{(1)} d \geq l_{(1)} \delta 
\text{ for any $i$ and $j$ s.t. } b_{ij} \neq 0
\end{array}
\right\},
\]
then we could have the equality $\tilde{\gamma} = \tilde{n}_{k-1}$
and the proposition above fails.
\end{rem}

%%%%%%%%%%%%%%%%%%%%%%%%%%%%%%%%%%%%%%%% second branched coverings for Case 4

Let $l_{(2)} = s_2/r_2$,
where $s_2$ and $r_2$ coprime positive integers, 
$\pi_2 (t, \mathsf{c}) = (t \mathsf{c}^{r_2}, \mathsf{c}^{s_2})$ 
and $\tilde{f}_2 = \pi_2^{-1} \circ \tilde{f}_1 \circ \pi_2$.
Then, formally,
\[
\tilde{f}_2 (t, \mathsf{c}) =
\left( \dfrac{\tilde{p}_1 (t \mathsf{c}^{r_2})}{\tilde{q}_1 (t \mathsf{c}^{r_2}, \mathsf{c}^{s_2})^{{r_2}/{s_2}}}, \ 
\tilde{q}_1(t \mathsf{c}^{r_2}, \mathsf{c}^{s_2})^{1/{s_2}} \right).
\]
Note that
$\pi_2^{-1} (\pi_1^{-1} (U^{l_{(1)},l_{(2)}})) 
= \{ |t \mathsf{c}^{(1 - 1/r_1) r_2}| > R^{1/r_1}, |\mathsf{c}| > R^{1/s_1} \}
\supset
%which includes
\{ |t| > R^{1/r_1}, |\mathsf{c}| > R^{1/s_1} \}$.
%= \{ |t^{r_1} \mathsf{c}^{(r_1 - 1)r_2}| > R, |\mathsf{c}^{s_1}| > R \}
%\supset \{ |t^{r_1}| > R, |\mathsf{c}^{s_1}| > R \}$.

\begin{prop} \label{} 
If $s_1/r_1 \in \mathcal{I}_f^1$, $s_2/r_2 \in \mathcal{I}_f^2$
and $\tilde{\gamma} /{s_2} \in \mathbb{N} \cup \{ 0 \}$,  
then $\tilde{f}_2$ is well defined and holomorphic 
on $\{ |t \mathsf{c}^{(1 - 1/r_1) r_2}| > R^{1/r_1}, |\mathsf{c}| > R^{1/s_1} \}$.
% $\{ |t^{r_1} \mathsf{c}^{(r_1 - 1)r_2}| > R, |\mathsf{c}^{s_1}| > R \}$.
% a neighborhood of the origin. 
More precisely,
\[
\tilde{f}_2 (t, \mathsf{c}) = 
\left( t^{\delta - r_2 \tilde{\gamma}/ s_2} 
 \mathsf{c}^{r_2 \delta - r_2 (r_2 \tilde{\gamma}/ s_2 + d)} 
 \{ 1 + \zeta_2 (t, \mathsf{c}) \},
\ t^{\tilde{\gamma}/ s_2} 
 \mathsf{c}^{r_2 \tilde{\gamma}/ s_2 + d} 
 \{ 1 + \eta_2 (t, \mathsf{c}) \} \right),
\]
where $\zeta_2$, $\eta_2 \to 0$ 
on $\{ |t \mathsf{c}^{(1 - 1/r_1) r_2}| > R^{1/r_1}, |\mathsf{c}| > R^{1/s_1} \}$
% $\{ |t^{r_1} \mathsf{c}^{(r_1 - 1)r_2}| > R, |\mathsf{c}^{s_1}| > R \}$ 
as $R \to \infty$.
\end{prop}

Recall that $\alpha_0 = \gamma/(\delta - d)$ and
let $\tilde{\alpha}_0 = \tilde{\gamma}/(\delta - d)$.
% \alpha_0 = \gamma/(\delta - d)

\begin{cor} \label{branched coverings: case4}
If $T_{k-1} < \delta \leq T_{k}$
and $s_1/r_1 = \alpha_0$,
then $\tilde{f}_2$ is well defined on the region above
for any $s_2/r_2$ in $\mathcal{I}_f^2$.
If $s_1/r_1 \in \mathcal{I}_f^1$ and 
$s_2/r_2 = \tilde{\alpha}_0$,
then $\tilde{f}_2$ is well defined on the region above.
% $\pi_2^{-1} (\pi_1^{-1} (U^{l_{(1)},l_{(2)}}))$.
\end{cor}

\begin{proof}
If $T_{k-1} < \delta$,
then 
% $\alpha_0 < l_1 + l_2$ and so $\mathcal{I}_f^1 = [l_1, \alpha_0]$. In particular,
$\alpha_0 \in \mathcal{I}_f^1$. 
Moreover, 
if $s_1/r_1 = \alpha_0$,
then $\tilde{\gamma} = 0$ and so 
%we obtain the condition 
$\tilde{\gamma} /s_2 = 0$.
%in the propoition above.
On the other hand,
if $s_2/r_2 = \tilde{\alpha}_0$,
then $\tilde{\gamma} /s_2 \in \mathbb{N}$.
\end{proof}

%%%%%%%%%%%%%%%%%%%%%%%%%%%%%%%%%%%%%%%%%%%%%%%%%%%%%%%%%%%%%%%%%%%%%%%%%%%%
%%%%%%%%%%%%%%%%%%%%%%%%%%%%%%%%%%%%%%%%%%%%%%%%%%%%%%%%%%%%%%%%%%%%%%%%%%%%
\section{Rational extensions}

In this section we illustrate that
$U$ is included in the attracting basin of a superattracting fixed or indeterminacy point at infinity,
or in the closure of the attracting basins of two point at infinity.
We first deal with the extension of $f$ to the projective space $\mathbb{P}^2$.
A polynomial map always extends to a rational map on $\mathbb{P}^2$.
We next deal with the extensions of $f$ to weighted projective spaces.
Although there is a condition for $f$ to extend a rational map on a weighted projective space,
it is useful to realize the rational extension 
whose dynamics on the line at infinity is induced by a polynomial
for the case $\delta > d \geq 2$ and $l = \alpha_0$ and
the case $\delta = d$, $\gamma = 0$ and $l = l_1$. 
We use the same notation $\tilde{f}$ for a extension of $f$
as the blow-up and the coverings of $f$.

Whereas similar descriptions for Cases 1 and 2 are given in \cite{ueno poly},
we improve % some parts such as 
the definition of 
the rational extension of $f$ to a weighted projective space and
state when it is well defined here.
One can also find arguments on extensions of polynomial maps to 
weighted projective spaces in Section 5.3 in \cite{fj poly}. 

%%%%%%%%%%%%%%%%%%%%%%%%%%%%%%%%%%%%%%%%%%%%%%%%%%%%%%%%%
\subsection{Projective space}

The projective space $\mathbb{P}^2$ is 
a quotient space of $\mathbb{C}^3 - \{ O \}$,
\[
\mathbb{P}^2 = \mathbb{C}^3 - \{ O \}/ \sim,
\]
where $(z,w,t) \sim (c z, c w, c t)$ for any $c$ in $\mathbb{C} - \{ 0 \}$.
%%%
The polynomial skew product $f$ extends to the rational map $\tilde{f}$ on $\mathbb{P}^2$,
\[
\tilde{f} [ z : w : t ] 
= \left[ p \left( \frac{z}{t} \right) t^{\lambda} 
: q \left( \frac{z}{t}, \frac{w}{t} \right) t^{\lambda} 
: t^{\lambda} \right],
\]
where $\lambda = \deg f = \max \{ \deg p, \deg q \}$.
By assumption, % the setting,
$\deg p \geq 2$ and $\deg q \geq 2$.
Let $L_{\infty}$ be the line at infinity
and $I_{\tilde{f}}$ the indeterminacy set of $\tilde{f}$.
Let $D = \deg q$ and
$h$ the sum of all the terms $b_{ij} z^i w^j$ in $q$
with the maximum degree $D$.
Let $b_{NM} z^N w^M$ and $b_{N^*M^*} z^{N^*} w^{M^*}$ be the terms in $h$
with the smallest and biggest degree with respect to $z$,
respectively.
Let $p_{\infty}^+ = [0:1:0]$ and $p_{\infty}^- = [1:0:0]$.
% We can roughly divide the dynamics of $\tilde{f}$ as follows.
% The following lemmas are helpful to
% investigate the dynamics of $\tilde{f}$ on $L_{\infty}$. % the line at infinity.

\begin{lem}\label{lem of trichotomy on proj sp}
We have the following trichotomy,
where  $u$ and $v$ are some polynomials.
\begin{enumerate}
\item %%%%
If $\delta < D$, then % $\lambda = D$ and
$\tilde{f} [ z : w : t ] 
= \left[ t^{D - \delta} \{ z^{\delta} + t u(z,t) \} : h(z,w) + t v(z,w,t) : t^{D} \right]$.
Hence $\tilde{f}$ collapses $L_{\infty} - I_{\tilde{f}}$ to $p_{\infty}^+$,
where $I_{\tilde{f}} = \{ [z:w:0] \, | \, h(z,w)=0 \}$.
\item %%%%
If $\delta = D$, then % $\lambda = \delta = D$ and
$\tilde{f} [ z : w : t ] 
= \left[ z^{\delta} + t u(z,t) : h(z,w) + t v(z,w,t) : t^{\delta} \right]$.
Hence the restriction of $\tilde{f}$ to $L_{\infty} - I_{\tilde{f}}$ 
is induced by $h$,
where $I_{\tilde{f}} \subset \{ p_{\infty}^+ \}$.
\item %%%%
If $\delta > D$, then % $\lambda = \delta$ and
$\tilde{f} [ z : w : t ] 
= \left[ z^{\delta} + t u(z,t) : t^{\delta - D} \{ h(z,w) + t v(z,w,t) \} : t^{\delta} \right]$.
Hence $\tilde{f}$ collapses $L_{\infty} - I_{\tilde{f}}$ to 
the superattracting fixed point $p_{\infty}^-$,
where $I_{\tilde{f}} = \{ p_{\infty}^+ \}$.
\end{enumerate}
For $(1)$ and $(2)$, 
$p_{\infty}^+$ is a superattracting fixed point if $N = 0$ and
an indeterminacy point if $N > 0$.
\end{lem}

\begin{lem}[Geometric characterization of $\lambda$]\label{lem of lambda on proj sp}
It follows that $\lambda$ coincides with 
the maximal $y$-intercept of the lines with slope $- 1$
that intersect with $\{ (0, \delta) \} \cup N(q)$.
\end{lem}

Let $z^{\gamma} w^d$ be a dominant term of $q$
and $U$ the corresponding region.
Let $A^+$ and $A^-$ be the attracting basins of 
$p_{\infty}^+$ and $p_{\infty}^-$,
respectively.
The notation $U \subset \overline{A^+ \cup A^-}$ 
in the propositions and tables below means that 
$U \subset A^+ \cup A^- \cup (\partial A^+ \cap \partial A^-)$ and
$U$ intersects both $A^+$ and $A^-$.
The following proposition gives a rough description of
the relation between $U$ and the attracting basins.

\begin{prop}\label{prop of rough trichotomy on proj sp}
We have the following rough classification. 
\begin{enumerate}
\item
If $\delta < D$, then $U \subset A^+$.
\item
If $\delta = D$ and $d \geq 2$,
then $U \subset A^+$, $U \subset \overline{A^+ \cup A^-}$ or $U \subset A^-$.
\item
If $\delta > D$, then $U \subset A^-$.
\end{enumerate}
More precisely,
let $\delta = D$ and $d \geq 2$.
\begin{enumerate}
\item[(4)]
If $\delta \neq T_k$ for any $k$,
then $h = z^{\gamma} w^d$ and $U \subset \overline{A^+ \cup A^-}$.
\item[(5)]
If $\delta = T_k$ for some $k$ and $\gamma > 0$,
then $U \subset A^+$ or $U \subset A^-$.
\item[(6)]
If $\delta = T_1$ and $\gamma = 0$,
then $U \subset A^+$ or $U \subset \overline{A^+ \cup A^-}$.
\end{enumerate}
\end{prop}

% This proposition follows from Propositions ... and ...
% To obtain these propositions,

Now we start to investigate the dynamics of $\tilde{f}$ on $L_{\infty}$ and 
the relation between $U$ and the attracting basins
more precisely case by case,
and obtain more detailed versions of the proposition above
as Propositions \ref{prop for delta > d on proj sp} and 
\ref{prop for delta =< d on proj sp}.

We first deal with Case 2.
Let $\delta \leq T_1$ and $(\gamma, d) = (n_1, m_1)$.
If $\delta > d$,
then $\gamma > 0$ and $\mathcal{I}_f = [ l_1, \alpha_0 ]$.
Moreover,
it follows from the shape of $N(q)$  and Lemma \ref{lem of lambda on proj sp} that
$\delta < \gamma + d \leq D$ if $\alpha_0 > 1$,
$\delta = \gamma + d = D$ if $\alpha_0 = 1$, and
$\delta > \gamma + d = D$ if $\alpha_0 < 1$
since the slope of the line passing through the points
$(0, \delta)$ and $(\gamma, d)$ is $- \alpha_0^{-1}$,
since $N(q)$ is included in the left-hand side of the line, and
since $N(q)$ intersects with the line at $(\gamma, d)$.
%The magnitude relation between $\alpha_0$ and $1$
%determines the magnitude relation between $\delta$ and $D$,
Therefore, 
using Lemma \ref{lem of trichotomy on proj sp},
we can classify the relation between $U$ and the attracting basins
as follows.
%we obtain the following classification table
%from Lemma ...
%Lemma ... induces
%the following relation between $U$ and the attracting basins.

%%%%%%%%%%%%%%%%%%%%%%%%%%%
\vspace{2mm}
\begin{center}
\begin{tabular}{|c|c|c|c|} \hline
 Case 2 & $\alpha_0 > 1$ & $\alpha_0 = 1$ & $\alpha_0 < 1$ \\ \hline
 $\delta > d$ & $\delta < D$ & $\delta = D$ & $\delta > D$ \\
 % $\delta > d$ & type 1 & type 2 & type 3 \\
 ( \& $\gamma > 0$) & $U \subset A^+$ & $U \subset A^+$ if $\delta = T_1$ and $d \geq 2$ & $U \subset A^-$ \\ 
 & & $U \subset \overline{A^+ \cup A^-}$ if $\delta < T_1$ and $d \geq 2$ & \\ \hline
\end{tabular}
\end{center}
\vspace{2mm}
%%%%%%%%%%%%%%%%%%%%%%%%%%%
Note that $p_{\infty}^+$ is always an indeterminacy point
since $N \geq \gamma > 0$.
On the other hand,
$p_{\infty}^-$ is a superattracting fixed point 
if $\alpha_0 = 1$ and $\delta < T_1$ or if $\alpha_0 < 1$.
If $\alpha_0 = 1$,
then $h$ contains $z^{\gamma} w^d$. % the dominant term 
Moreover,
$h$ contains other terms such as $b_{n_2 m_2} z^{n_2} w^{m_2}$ if $\delta = T_1$, and % b_{n_2 m_2} 
$h = z^{\gamma} w^d$ if $\delta < T_1$.

If $\delta \leq d$,
then $\mathcal{I}_f = [ l_1, \infty)$.
For the case $\delta \leq d$ and $\gamma > 0$ and
the case $\delta < d$ and $\gamma = 0$,
it follows that $\delta < \gamma + d < D$ if $l_1 > 1$, and 
$\delta < \gamma + d = D$ if $l_1 \geq 1$.
On the other hand,
for the case $\delta = d$ and $\gamma = 0$, 
it follows that 
$\delta = \gamma + d < D$ if $l_1 > 1$, and
$\delta = \gamma + d = D$ if $l_1 \leq 1$.
%Consequently,
Combining these cases,
we obtain the following classification table.
%Lemma ... induces the following.

%%%%%%%%%%%%%%%%%%%%%%%%%%%
\vspace{2mm}
\begin{center}
\begin{tabular}{|c|c|c|c|} \hline
 Case 2 & $l_1 > 1$ & $l_1 = 1$ & $l_1 < 1$ \\ \hline
  $\delta \leq d$ \& $\gamma > 0$ & $\delta < D$ & $\delta < D$ & $\delta < D$ \\
 % $\delta \leq d$ \& $\gamma > 0$ & type 1 & type 1 & type 1 \\ 
 $\delta < d$ \& $\gamma = 0$ & $U \subset A^+$ & $U \subset A^+$ & $U \subset A^+$ \\ \hline
 $\delta = d$ \& $\gamma = 0$ & $\delta < D$ & $\delta = D$ & $\delta = D$ \\
 % $\delta = d$ \& $\gamma = 0$ & type 1 & type 2 & type 2 \\ 
  & $U \subset A^+$ & $U \subset A^+$ & $U \subset \overline{A^+ \cup A^-}$ \\ \hline
\end{tabular}
\end{center}
\vspace{2mm}
%%%%%%%%%%%%%%%%%%%%%%%%%%%
Note that 
$p_{\infty}^+$ is a superattracting fixed point if $\delta < d$, $\gamma = 0$ and $l_1 \leq 1$ and
an indeterminacy point otherwise.
If $\delta = d$, $\gamma = 0$ and $l_1 \leq 1$,
then $\tilde{f}$ is holomorphic and  
$h$ contains $z^{\gamma} w^d$.
Moreover,
$h$ contains other terms such as $b_{n_2 m_2} z^{n_2} w^{m_2}$ if $l_1 = 1$, and % b_{n_2 m_2}
$h = w^d$ and $p_{\infty}^-$ is a superattracting fixed point if $l_1 < 1$.

\begin{rem}
For the case $\delta \leq d$ and $\gamma > 0$ and
the case $\delta < d$ and $\gamma = 0$,
it might be useful to regard the branch point $\alpha_0$ as $\infty$.
For the case $\delta = d$ and $\gamma = 0$,
although $\alpha_0$ is not well defined for $(\gamma, d) = (n_1, m_1)$,
it is well defined for the next vertex $(n_2, m_2)$,
which coincides with $l_1$.
\end{rem}

We next deal with Case 3.
Let $\delta \geq T_{s-1}$ and $(\gamma, d) = (n_s, m_s)$.
% Recall that $\mathcal{I}_f = [ \alpha_0, l_2 ]$.
Since $\delta > d$ and $\gamma > 0$,
the classification table below is similar to the case $\delta > d$ for Case 2,
but not the same.
%the relation between $U$ and the attracting basins is similar to
%the case $\delta > d$ in Case 2.

%%%%%%%%%%%%%%%%%%%%%%%%%%%
\vspace{2mm}
\begin{center}
\begin{tabular}{|c|c|c|c|} \hline
 Case 3 & $\alpha_0 > 1$ & $\alpha_0 = 1$ & $\alpha_0 < 1$ \\ \hline
 $\delta > d$ & $\delta < D$ & $\delta = D$ & $\delta > D$ \\
 ( \& $\gamma > 0$) & $U \subset A^+$ & $U \subset A^-$ if $\delta = T_{s-1}$ and $d \geq 2$ & $U \subset A^-$ \\ 
  & & $U \subset \overline{A^+ \cup A^-}$ if $\delta > T_{s-1}$ and $d \geq 2$ & \\ \hline
 %$\gamma = 0$ & & &  $V_R \subset A_{[1:0:0]}$ \\ \hline
\end{tabular}
\end{center}
\vspace{2mm}
%%%%%%%%%%%%%%%%%%%%%%%%%%%
More precisely,
$\delta < \gamma + d = D$ if $\alpha_0 > 1$,
$\delta = \gamma + d = D$ if $\alpha_0 = 1$, and
$\delta > D \geq \gamma + d$ if $\alpha_0 < 1$.
Note that
$p_{\infty}^+$ is always an indeterminacy point and
$p_{\infty}^-$ is a superattracting fixed point if $\alpha_0 \leq 1$.
If $\alpha_0 = 1$,
then $h$ contains $z^{\gamma} w^d$.
Moreover,
$h$ contains $b_{n_{s-1} m_{s-1}} z^{n_{s-1}} w^{m_{s-1}}$ if $\delta = T_{s-1}$, and % b_{n_{s-1} m_{s-1}} 
$h = z^{\gamma} w^d$ if $\delta > T_{s-1}$.

We finally deal with Case 4.
Let $T_{k-1} \leq \delta \leq T_k$ and $(\gamma, d) = (n_k, m_k)$.
Since $\delta > d$ and $\gamma > 0$,
the classification table below is again similar to the case $\delta > d$ for Case 2.

%%%%%%%%%%%%%%%%%%%%%%%%%%%
\vspace{2mm}
\begin{center}
\begin{tabular}{|c|c|c|c|} \hline
 Case 4 & $\alpha_0 > 1$ & $\alpha_0 = 1$ & $\alpha_0 < 1$ \\ \hline
 $\delta > d$ & $\delta < D$ & $\delta = D$ & $\delta > D$ \\
 ( \& $\gamma > 0$) & $U \subset A^+$ & $U \subset A^+$ if $\delta = T_k$ and $d \geq 2$ & $U \subset A^-$ \\ 
 & & $U \subset A^-$ if $\delta = T_{k-1}$ and $d \geq 2$ & \\ 
 & & $U \subset \overline{A^+ \cup A^-}$ if $T_{k-1} < \delta < T_k$ and $d \geq 2$ & \\ \hline
\end{tabular}
\end{center}
\vspace{2mm}
%%%%%%%%%%%%%%%%%%%%%%%%%%%
More precisely,
$\delta < \gamma + d \leq D$ if $\alpha_0 > 1$,
$\delta = \gamma + d = D$ if $\alpha_0 = 1$, and
$\delta > D \geq \gamma + d$ if $\alpha_0 < 1$.
Note that
$p_{\infty}^+$ is always an indeterminacy point and
$p_{\infty}^-$ is a superattracting fixed point if $\alpha_0 = 1$ and $\delta < T_k$ or if $\alpha_0 < 1$.
If $\alpha_0 = 1$,
then $h$ contains $z^{\gamma} w^d$.
Moreover,
$h$ contains $b_{n_{k+1} m_{k+1}} z^{n_{k+1}} w^{m_{k+1}}$ if $\delta = T_{k}$, % b_{n_{k+1} m_{k+1}}
$h$ contains $b_{n_{k-1} m_{k-1}} z^{n_{k-1}} w^{m_{k-1}}$ if $\delta = T_{k-1}$, and % b_{n_{k-1} m_{k-1}} 
$h = z^{\gamma} w^d$ if $T_{k-1} < \delta < T_k$.

Consequently,
we obtain the following two propositions,
which implies Proposition \ref{prop of rough trichotomy on proj sp}.

\begin{prop}\label{prop for delta > d on proj sp}
Let $\delta > d$.
Then $\gamma > 0$ and so $\alpha_0 > 0$.
\begin{enumerate}
\item
If $\alpha_0 > 1$,
then $\delta < D$ and $U \subset A^+$. 
\item 
If $\alpha_0 = 1$, $d \geq 2$ and $\delta \neq T_k$ for any $k$,
then $\delta = D$, $h = z^{\gamma} w^d$ and $U \subset \overline{A^+ \cup A^-}$.
\item
If $\alpha_0 < 1$,
then $\delta > D$ and $U \subset A^-$. 
\end{enumerate}
Moreover,
let $\alpha_0 = 1$, $d \geq 2$ and $\delta = T_k$ for some $k$.
Then $\delta = D$ and
$h$ contains $b_{NM} z^N w^M$ and $b_{N^*M^*} z^{N^*} w^{M^*}$.
\begin{enumerate}
\item[(4)]
If $(\gamma, d) = (N, M)$,
then % $\delta = D$ and 
$U \subset A^+$. 
\item[(5)]
If $(\gamma, d) = (N^*, M^*)$,
then %$\delta = D$ and 
$U \subset A^-$.
\end{enumerate}
For all the cases,
$p_{\infty}^+$ is an indeterminacy point and
$p_{\infty}^-$ is a superattracting fixed point
if  $U \cap A^+ \neq \emptyset$ and $U \cap A^- \neq \emptyset$,
respectively.
\end{prop}

\begin{prop}\label{prop for delta =< d on proj sp}
Let $\delta \leq d$.
Then $(\gamma, d)$ belongs to Case 2.
\begin{enumerate}
\item
If $\gamma > 0$,
then $\delta < D$ and $U \subset A^+$, 
where $p_{\infty}^+$ is an indeterminacy point.
\item
If $\delta < d$ and $\gamma = 0$,
then $\delta < D$ and $U \subset A^+$, 
where $p_{\infty}^+$ is a superattracting fixed point if $l_1 \leq 1$ and
an indeterminacy point if $l_1 > 1$.
\end{enumerate}
Moreover,
let $\delta = d$ and $\gamma = 0$.
Then $\delta = T_1$, and
$\tilde{f}$ is holomorphic if $l_1 \leq 1$.
\begin{enumerate}
\item[(3)]
If $l_1 > 1$,
then $\delta < D$ and $U \subset A^+$, 
where $p_{\infty}^+$ is an indeterminacy point.
\item[(4)]
If $l_1 = 1$,
then $\delta = D$ and $U \subset A^+$,
where $p_{\infty}^+$ is a superattracting fixed point.
\item[(5)]
If $l_1 < 1$,
then $\delta = D$, $h = z^{\gamma} w^d$ and $U \subset \overline{A^+ \cup A^-}$,
where  $p_{\infty}^+$ and $p_{\infty}^-$ are superattracting fixed points.
\end{enumerate}
\end{prop}

%%%%%%%%%%%%%%%%%%%%%%%%%%%%%%%%%%%%%%%%%%%%%%%%%%%%%%%%%
\subsection{Weighted projective spaces}

Let $r$ and $s$ be coprime positive integers.
The weighted projective space $\mathbb{P} (r,s,1)$ is 
a quotient space of $\mathbb{C}^3 - \{ O \}$,
\[
\mathbb{P} (r,s,1) = \mathbb{C}^3 - \{ O \}/ \sim,
\]
where $(z,w,t) \sim (c^r z, c^s w, c t)$ for any $c$ in $\mathbb{C} - \{ 0 \}$.
Let us again denote the weighted homogeneous coordinate as $[ z : w : t ]$
for simplicity.
Let $l = s/r$ and
\[
D_l = \max \left\{ \ l^{-1} i + j \ | % i \text{ and } j \text{ s.t. }  b_{ij} \neq 0 \right\}.
\begin{array}{lcr}
i \text{ and } j \text{ s.t. }  b_{ij} \neq 0
\end{array}
\right\}.
\]
For a polynomial skew product $f$,
we define
\[
\tilde{f} [ z : w : t ] 
= \left[ p \left( \frac{z}{t^r} \right) t^{\lambda r} 
: q \left( \frac{z}{t^r}, \frac{w}{t^s} \right) t^{\lambda s} 
: t^{\lambda} \right],
\]
where $\lambda_l = \max \{ \deg p, D_l \}$.
Note that $\lambda_l = \deg p = \delta$ or
$\lambda_l = l^{-1} n_j + m_j$ 
%= \delta + \dfrac{r N + s M - s \delta}{s}
for some vertex $(n_j, m_j)$ of $N(q)$.
For the later case,
if $n_j/s$ is integer, then so is $\lambda_l$.

\begin{lem}
If $\lambda_l$ is integer,
then every components of $\tilde{f}$ are polynomial.
Hence $\tilde{f}$ is well defined and rational on $\mathbb{P} (r,s,1)$.
\end{lem}

We use the same notations 
$L_{\infty}$, $I_{\tilde{f}}$, $p_{\infty}^{\pm}$, $A^{\pm}$,
$h$, $(N,M)$ and $(N^*, M^*)$
as the projective space case.
% The following lemmas are helpful to
% investigate the dynamics of $\tilde{f}$ on $L_{\infty}$. % the line at infinity.

\begin{lem}
We have the following trichotomy,
where $u$ and $v$ are some polynomials.
\begin{enumerate}
\item %%%%%%
If $\delta < D_l$ and $\lambda_l$ is integer, then 
$\tilde{f} [ z : w : t ] 
= \left[ t^{\lambda - \delta} \{ z^{\delta} + t u(z,t) \} : h(z,w) + t v(z,w,t) : t^{\lambda} \right]$.
Hence $\tilde{f}$ collapses $L_{\infty} - I_{\tilde{f}}$ to $p_{\infty}^+$,
where $I_{\tilde{f}} = \{ [z:w:0] \, | \, h(z,w)=0 \}$.
\item %%%%%%
If $\delta = D_l$, then
$\tilde{f} [ z : w : t ] 
= \left[ z^{\delta} + t u(z,t) : h(z,w) + t v(z,w,t) : t^{\delta} \right]$.
Hence the restriction of $\tilde{f}$ to $L_{\infty} - I_{\tilde{f}}$ 
is induced by $h$,
where $I_{\tilde{f}} \subset \{ p_{\infty}^+ \}$.
\item %%%%%%
If $\delta > D_l$, then
$\tilde{f} [ z : w : t ] 
= \left[ z^{\delta} + t u(z,t) : t^{s (\delta - D_l)} \{ h(z,w) + t v(z,w,t) \} : t^{\delta} \right]$.
Hence $\tilde{f}$ collapses $L_{\infty} - I_{\tilde{f}}$ to 
the superattracting fixed point $p_{\infty}^-$,
where $I_{\tilde{f}} = \{ p_{\infty}^+ \}$.
\end{enumerate}
For $(1)$ and $(2)$,
$p_{\infty}^+$ is a superattracting fixed point if $N = 0$ and
an indeterminacy point if $N > 0$.
\end{lem}

\begin{lem}[Geometric characterization of $\lambda_l$]
It follows that $\lambda_l$ coincides with 
the maximal $y$-intercept of the lines with slope $- l^{-1}$
that intersect with $\{ (0, \delta) \} \cup N(q)$.
\end{lem}

\begin{rem}[Geometric characterization of $\alpha_0$]
Let $\delta > T_1$.
Then $\alpha_0$ coincides with
\[
\min 
\left\{ \ l > 0 \ | 
\begin{array}{lcr}
l \delta \geq i + l j
\text{ for any $i$ and $j$ s.t. } b_{ij} \neq 0
\end{array}
\right\}
\]
as described in Section 3 in \cite{ueno poly}. %  mentioned
% as pointed out in \cite{ueno poly}.
In other words,
$- \alpha_0^{-1}$ coincides with the slope of the line 
that intersects with both $\{ (0, \delta) \}$ and 
the boundary of $N(q)$ % $\partial N(q)$ 
but does not intersect with 
the interior of $N(q)$. % $int N(q)$. 
\end{rem}

%%%%%%%%%%%%%%%%%%%%%%%%%%%%%%%%%%%%%%%%%%%%% 

Let $z^{\gamma} w^d$ be a dominant term of $q$
and $U$ the corresponding region.
% $(\gamma, d)$ be the vertex of $N(q)$.
The dynamics of $\tilde{f}$ on $L_{\infty}$ and 
the relation between $U$ and the attracting basins are % $A^{\pm}$ are 
almost the same as the projective space case:
as shown in the following tables and propositions,
we only need to change $1$ to $l$
in comparison with $\alpha_0$ or $l_1$,
to change $D$ to $D_l$
in comparison with $\delta$, and
to add the condition $\lambda_l \in \mathbb{N}$ 
when $l < \alpha_0$ or $l < l_1$.

%Hence we obtain the following tables.

We first exhibit classification tables and a proposition
for the case $\delta > d$,
which are obtained by similar arguments as the projective space case.
%More presicely, for example,
Note that
$\delta < l^{-1} \gamma + d \leq D_l$ if $l < \alpha_0$, 
$\delta = l^{-1} \gamma + d = D_l$ if $l = \alpha_0$, and
$\delta > D_l \geq l^{-1} \gamma + d$ if $l > \alpha_0$.

%%%%%%%%%%%%%%%%%%%%%%%%%%%
\vspace{2mm}
\begin{center}
\begin{tabular}{|c|c|c|c|} \hline
 Case 2 & $l < \alpha_0$ & $l = \alpha_0$ & $l > \alpha_0$ \\ \hline
 $\delta > d$ & $\delta < D_l$ & $\delta = D_l$ & $\delta > D_l$ \\
 % $\delta > d$ & type 1 & type 2 & type 3 \\
 ( \& $\gamma > 0$) & $U \subset A^+$ if $\lambda_l \in \mathbb{N}$ & $U \subset A^+$ if $\delta = T_1$ and $d \geq 2$ & $U \subset A^-$ \\ 
 & & $U \subset \overline{A^+ \cup A^-}$ if $\delta < T_1$ and $d \geq 2$ & \\ \hline
\end{tabular}
\end{center}
%%%%%%%%%%%%%%%%%%%%%%%%%%%

%%%%%%%%%%%%%%%%%%%%%%%%%%%
\vspace{1mm}
\begin{center}
\begin{tabular}{|c|c|c|c|} \hline
 Case 3 & $l < \alpha_0$ & $l = \alpha_0$ & $l > \alpha_0$ \\ \hline
 $\delta > d$ & $\delta < D_l$ & $\delta = D_l$ & $\delta > D_l$ \\
 ( \& $\gamma > 0$) & $U \subset A^+$ if $\lambda_l \in \mathbb{N}$ & $U \subset A^-$ if $\delta = T_{s-1}$ and $d \geq 2$ & $U \subset A^-$ \\ 
  & & $U \subset \overline{A^+ \cup A^-}$ if $\delta > T_{s-1}$ and $d \geq 2$ & \\ \hline
 %$\gamma = 0$ & & &  $V_R \subset A_{[1:0:0]}$ \\ \hline
\end{tabular}
\end{center}
%%%%%%%%%%%%%%%%%%%%%%%%%%%

%%%%%%%%%%%%%%%%%%%%%%%%%%%
\vspace{1mm}
\begin{center}
\begin{tabular}{|c|c|c|c|} \hline
 Case 4 & $l < \alpha_0$ & $l = \alpha_0$ & $l > \alpha_0$ \\ \hline
 $\delta > d$ & $\delta < D_l$ & $\delta = D_l$ & $\delta > D_l$ \\
 ( \& $\gamma > 0$) & $U \subset A^+$ if $\lambda_l \in \mathbb{N}$ & $U \subset A^+$ if $\delta = T_k$ and $d \geq 2$ & $U \subset A^-$ \\ 
 & & $U \subset A^-$ if $\delta = T_{k-1}$ and $d \geq 2$ & \\ 
 & & $U \subset \overline{A^+ \cup A^-}$ if $T_{k-1} < \delta < T_k$ and $d \geq 2$ & \\ \hline
\end{tabular}
\end{center}
\vspace{2mm}
%%%%%%%%%%%%%%%%%%%%%%%%%%%

%Consequentry,
%we obtain the following proposition.

\begin{prop}
Let $\delta > d$.
Then $\gamma > 0$ and so $\alpha_0 > 0$.
\begin{enumerate}
\item %%%%%%
If $l < \alpha_0$ and  $\lambda_l$ is integer, 
then $\delta < D_l$ and $U \subset A^+$. 
\item %%%%%%
If $l = \alpha_0$, $d \geq 2$ and $\delta \neq T_k$ for any $k$, 
then $\delta = D_l$, $h = z^{\gamma} w^d$ and $U \subset \overline{A^+ \cup A^-}$.
\item %%%%%%
If $l > \alpha_0$, 
then $\delta > D_l$ and $U \subset A^-$. 
\end{enumerate}
Moreover,
let $l = \alpha_0$, $d \geq 2$ and $\delta = T_k$ for some $k$.
Then $\delta = D_l$ and $h$ contains $b_{NM} z^N w^M$ and $b_{N^*M^*} z^{N^*} w^{M^*}$.
\begin{enumerate}
\item[(4)] %%%%%%
If $(\gamma, d) = (N, M)$,
then $U \subset A^+$. 
\item[(5)] %%%%%%
If $(\gamma, d) = (N^*, M^*)$,
then $U \subset A^-$.
\end{enumerate}
For all the cases,
$p_{\infty}^+$ is an indeterminacy point and
$p_{\infty}^-$ is a superattracting fixed point
if  $U \cap A^+ \neq \emptyset$ and $U \cap A^- \neq \emptyset$,
respectively.
\end{prop}

We next exhibit a classification table and a proposition
for the case $\delta \leq d$.

%%%%%%%%%%%%%%%%%%%%%%%%%%%
\vspace{2mm}
\begin{center}
\begin{tabular}{|c|c|c|c|} \hline
 Case 2 & $l < l_1$ & $l = l_1$ & $l > l_1$ \\ \hline
  $\delta \leq d$ \& $\gamma > 0$ & $\delta < D_l$ & $\delta < D_l$ & $\delta < D_l$ \\
 % $\delta \leq d$ \& $\gamma > 0$ & type 1 & type 1 & type 1 \\ 
 $\delta < d$ \& $\gamma = 0$ & $U \subset A^+$ if $\lambda_l \in \mathbb{N}$ & $U \subset A^+$ & $U \subset A^+$ \\ \hline
 $\delta = d$ \& $\gamma = 0$ & $\delta < D_l$ & $\delta = D_l$ & $\delta = D_l$ \\
 % $\delta = d$ \& $\gamma = 0$ & type 1 & type 2 & type 2 \\ 
  & $U \subset A^+$ if $\lambda_l \in \mathbb{N}$ & $U \subset A^+$ & $U \subset \overline{A^+ \cup A^-}$ \\ \hline
\end{tabular}
\end{center}
\vspace{2mm}
%%%%%%%%%%%%%%%%%%%%%%%%%%%

\begin{prop}
Let $\delta \leq d$.
Then $(\gamma, d)$ belongs to Case 2.
\begin{enumerate}
\item
If $\gamma > 0$ and  $\lambda$ is integer,
then $\delta < D_l$ and $U \subset A^+$, 
where $p_{\infty}^+$ is an indeterminacy point.
\item
If $\delta < d$, $\gamma = 0$ and  $\lambda_l$ is integer,
then $\delta < D_l$ and $U \subset A^+$, 
where $p_{\infty}^+$ is a superattracting fixed point if $l \geq l_1$ and
an indeterminacy point if $l < l_1$.
\end{enumerate}
Moreover,
let $\delta = d$ and $\gamma = 0$.
Then $\delta = T_1$, and
$\tilde{f}$ is holomorphic if $l \geq l_1$.
\begin{enumerate}
\item[(3)]
If $l < l_1$ and  $\lambda_l$ is integer,
then $\delta < D_l$ and $U \subset A^+$, 
where $p_{\infty}^+$ is an indeterminacy point.
\item[(4)]
If $l = l_1$,
then $\delta = D_l$ and $U \subset A^+$,
where $p_{\infty}^+$ is a superattracting fixed point.
\item[(5)]
If $l > l_1$,
then $\delta = D_l$, $h = z^{\gamma} w^d$ and $U \subset \overline{A^+ \cup A^-}$,
where  $p_{\infty}^+$ and $p_{\infty}^-$ are superattracting fixed points.
\end{enumerate}
\end{prop}

%%%%%%%%%%%%%%%%%%%%%%%%%%%%%%%%%%%%%%%%%%%%%%%%%%%%%%%%%%%%%%%%%%%%%%%%%%%%
%%%%%%%%%%%%%%%%%%%%%%%%%%%%%%%%%%%%%%%%%%%%%%%%%%%%%%%%%%%%%%%%%%%%%%%%%%%%
\section{Proof of main theorem}

Theorem \ref{main theorem} follows from Proposition \ref{main lemma}
by almost the same arguments as in \cite{ueno SA},
which are described again for the completeness. % denoted
We first prove that
the composition $\phi_n = f_0^{-n} \circ f^n$ is well defined on $U$ in Section 7.1 and
converges uniformly to $\phi$ on $U$ if $d \geq 2$ in Section 7.2.
To prove the convergence,
we lift $f$ by the exponential product. 
We next prove that $\phi_n$ converges uniformly to $\phi$ on $U$ 
even if $d = 1$ and $\delta \neq T_k$ for any $k$ in Section 7.3.
To prove this,
we need more precise estimates.
Example 7.5 in \cite{ueno SA} shows that 
we cannot remove the condition $\delta \neq T_k$ for any $k$.
Finally,
we prove that $\phi$ is injective on $U$ in Section 7.4.
In Sections 7.3 and 7.4
we need to adapt the definition of $M$ and regions % modify ?
to the case of  polynomial skew products.

%%%%%%%%%%%%%%%%%%%%%%%%%%%%%%%%%%%%%%%%%%%%%%%%%%%%%%%%%
\subsection{Well definedness of $\phi_n$}

Thanks to Proposition \ref{main lemma},
we may write 
\[
p(z) = a_{\delta} z^{\delta} \{ 1 + \zeta (z) \} \text{ and } 
q(z,w) = b_{\gamma d} z^{\gamma} w^d \{ 1 + \eta (z,w) \},
\]
where $\zeta$ and $\eta$ are holomorphic on $U$ and
converge to $0$ on $U$ as $R \to \infty$.
We assume that $a_{\delta} = 1$ and $b_{\gamma d} = 1$ for simplicity.
Then the first and second components of $f^n$ are written as
\begin{align*}
& z^{\delta^n} \prod_{j = 1}^{n} \{ 1 + \zeta (p^{j - 1} (z)) \}^{\delta^{n-j}}
\text{ and} \\ 
& z^{\gamma_n} w^{d^n} 
\prod_{j = 1}^{n - 1} \{ 1 + \zeta (p^{j - 1} (z)) \}^{\gamma_{n-j}} 
\prod_{j = 1}^{n} \{ 1 + \eta (f^{j - 1} (z,w)) \}^{d^{n-j}}, 
\end{align*}
where $\gamma_n = \sum_{j=1}^{n} \delta^{n-j} d^{j-1} \gamma$.
We remark that the coefficients of the dominant terms
$z^{\delta^n}$ and $z^{\gamma_n} w^{d^n}$ are exactly
$a_{\delta}^{\delta^{n-1} + \cdots + \delta + 1}$ and
$a_{\delta}^{\gamma_{n-1} + \cdots + \gamma_2 + \gamma_1} 
  b_{\gamma d}^{d^{n-1}  + \cdots + d + 1}$,
respectively. 

Since
$f_0^{-n} (z,w) = (z^{1 / \delta^n}, z^{- \gamma_n / \delta^n d^n} w^{1 / d^n})$,
we can define $\phi_n$ as % follows:
\[ 
%\phi_n (z,w) = 
\left( z \cdot \prod_{j = 1}^{n} \sqrt[\delta^j]{1 + \zeta (p^{j - 1} (z))},
w \cdot \prod_{j = 1}^{n} \frac{\sqrt[d^j]{1 + \eta (f^{j - 1} (z,w))}}
{\sqrt[(\delta d)^j]{\{ 1 + \zeta (p^{j - 1} (z)) \}^{\gamma_j}}} \right),
\]
which is well defined and so holomorphic on $U$.
%%%

%%%%%%%%%%%%%%%%%%%%%%%%%%%%%%%%%%%%%%%%%%%%%%%%%%%%%%%%%
\subsection{Uniform convergence of $\phi_n$ when $d \geq 2$}

In order to prove the uniform convergence of $\phi_n$,
we lift $f$ and $f_0$ to $F$ and $F_0$ 
by the exponential product $\pi (z,w) = (e^z, e^w)$; 
that is, 
$\pi \circ F = f \circ \pi$ and $\pi \circ F_0 = f_0 \circ \pi$.
More precisely,
we define
\[
F(Z, W) = (P(Z), Q(Z,W)) 
= \left( \delta Z + \log \{ 1 + \zeta (e^Z) \},
\gamma Z + dW + \log \{ 1 + \eta (e^Z, e^W) \} \right)
\]
and $F_0 (Z,W) = (\delta Z, \gamma Z + d W)$.
By Proposition \ref{main lemma},
we may assume that
$\| F - F_0 \| < \tilde{\varepsilon} \text{ on } \pi^{-1} (U)$,
where $||(Z,W)|| = \max \{ |Z|, |W| \}$ and $\tilde{\varepsilon} = \log (1 + \varepsilon)$.
Similarly, 
we can lift $\phi_n$ to $\Phi_n$ so that 
the equation $\Phi_n = F_0^{-n} \circ F^n$ holds; thus, for any $n \geq 1$, 
\[
\Phi_n (Z, W) = \left( \frac{1}{\delta^n} P_n(Z), 
\frac{1}{d^n} Q_n(Z,W) - \frac{\gamma_n}{\delta^n d^n} P_n(Z) \right), 
\] 
where $(P_n(Z),Q_n(Z,W)) = F^n(Z,W)$.
Let $\Phi_n = (\Phi_n^1, \Phi_n^2)$. 
Then
\begin{align*}
|\Phi_{n+1}^1 - \Phi_n^1|
&= \left| \frac{P_{n+1}}{\delta^{n+1}} - \frac{P_n}{\delta^n} \right|
= \frac{|P_{n+1} - \delta P_n|}{\delta^{n+1}} 
< \frac{1}{\delta^{n+1}} \tilde{\varepsilon}
\text{ and } \\
|\Phi_{n+1}^2 - \Phi_n^2|
&= \left| \left\{ \frac{Q_{n+1}}{d^{n+1}} - 
 \frac{\gamma_{n+1} P_{n+1}}{\delta^{n+1} d^{n+1}} \right\}
 - \left\{ \frac{Q_n}{d^n} - \frac{\gamma_n P_n}{\delta^n d^n} \right\} \right| \\
&= \left| \frac{Q_{n+1}}{d^{n+1}} 
 - \frac{\gamma P_n}{d^{n+1}} - \frac{Q_n}{d^n} \right| 
 + \left| \frac{\gamma_{n+1} P_{n+1}}{\delta^{n+1} d^{n+1}}  
 - \frac{\gamma_n P_n}{\delta^n d^n} - \frac{\gamma P_n}{d^{n+1}} \right| \\
&= \frac{|Q_{n+1} - (\gamma P_n + d Q_n)|}{d^{n+1}} 
 + \frac{\gamma_{n+1} |P_{n+1} - \delta P_n|}{\delta^{n+1} d^{n+1}} 
< \frac{1}{d^{n+1}} \tilde{\varepsilon} + \frac{\gamma_{n+1}}{\delta^{n+1} d^{n+1}} \tilde{\varepsilon}.
\end{align*}
Hence $\Phi_n$ converges uniformly to $\Phi$ if $d \geq 2$.
In particular, 
\begin{align*}
\| \Phi - id \| & < \max \left\{ \frac{1}{\delta - 1}, 
\frac{1}{d-1} + \frac{\gamma}{\delta - d} 
\left( \frac{1}{d - 1} - \frac{1}{\delta - 1} \right) \right\} \tilde{\varepsilon}
\text{ if } \delta \neq d, \text{ and} \\
\| \Phi - id \| & < \left\{ \frac{1}{d - 1} + \frac{\gamma}{(d-1)^2} \right\} \tilde{\varepsilon}
\text{ if } \delta = d.
\end{align*}
By the inequality $|e^{z_1}/e^{z_2} - 1| \leq |z_1 - z_2| e^{|z_1 - z_2|}$,
the uniform convergence of $\Phi_n$ induces that of $\phi_n$. % translates into
Therefore,
$\phi$ is holomorphic on $U$.
In particular, if $||\Phi - id|| < \varepsilon$, 
then $|\phi_1 - z| < \varepsilon e^{\varepsilon} |z|$ and 
$|\phi_2 - w| < \varepsilon e^{\varepsilon} |w|$,
where $\phi = (\phi_1, \phi_2)$.
Hence $\phi \sim id$ on $U$ as $R \to \infty$.

%%%%%%%%%%%%%%%%%%%%%%%%%%%%%%%%%%%%%%%%%%%%%%%%%%%%%%%%%
\subsection{Uniform convergence of $\phi_n$ when $d = 1$}

We have proved the invariance of $U$ in Proposition \ref{main lemma}.
More strongly,
$f^n$ contracts $U$ rapidly. % , and
% the following lemma is the beginning of the proof of 
% the uniform convergence of $\phi_n$ when $d = 1$.

\begin{lem}\label{lem1: d = 1}
If $d = 1$ and $\delta \neq T_k$ for any $k$, 
then $f^n(U_R) \subset U_{2^n R}$ for large $R$.
\end{lem}

\begin{proof}
It is enough to show the lemma for Case 4.
We first give an abstract idea of the proof.
Recall that % If $b = 1$ then, formally,
\[
\tilde{f}_2 (t,c) \sim 
( t^{\delta - l_2^{-1} \tilde{\gamma}} c^{l_2^{-1} (\delta - \tilde{d})}, \ t^{\tilde{\gamma}} c^{\tilde{d}} )
% \text{ on } \{ |t| > R, |c| > R \} \text{ as } R \to \infty.
\]
on $\{ |t| > R, |c| > R \}$ as $R \to \infty$.
By assumption,
$\delta - l_2^{-1} \tilde{\gamma} > d = 1$,
$\delta - \tilde{d} > 0$,
$\tilde{\gamma} > 0$ and 
$\tilde{d} > d = 1$,
where
$\tilde{\gamma} = \gamma + l_1 d - l_1 \delta$ and
$\tilde{d} = l_2^{-1} \tilde{\gamma} + d$.
If $\tilde{f}_2$ is well defined,
then it is easy to check that
$\tilde{f}_2( \{ |t| > R, |c| > R \} ) \subset \{ |t| > 2R, |c| > 2R \}$ and so
$\tilde{f}_2^n( \{ |t| > R, |c| > R \} ) \subset \{ |t| > 2^n R, |c| > 2^n R \}$.

This idea provides a proof immediately.
Actually,
\begin{align*}
& \left| \frac{p(z)^{1 + l_1 l_2^{-1}}}{q(z,w)^{l_2^{-1}}} \right| 
> C_1 \left| t^{\delta - l_2^{-1} \tilde{\gamma}} c^{l_2^{-1} (\delta - \tilde{d})} \right|
> C_1 |t|^{\delta - l_2^{-1} \tilde{\gamma} - 1} |c|^{l_2^{-1} (\delta - \tilde{d})} \cdot |t|
> 2R 
\text{ and} \\
& \left| \frac{q(z,w)}{p(z)^{l_1}} \right| 
> C_2 \left| t^{\tilde{\gamma}} c^{\tilde{d}} \right|
> C_2 |t|^{\tilde{\gamma}} |c|^{\tilde{d} - 1} \cdot |c| 
> 2R
\end{align*}
for some constants $C_1$ and $C_2$ and for large $R$.
Hence 
$f(U_R) \subset U_{2R}$ and so 
$f^n(U_R) \subset U_{2^n R}$.
\end{proof}

Let $M = 1$ for Cases 1, 2 and 3 and 
$M = \min \{ \min \{ \tilde{\gamma} - \tilde{i} \ | \, \tilde{\gamma} >  \tilde{i} \text{ and } b_{ij} \neq 0 \}, 1 \}$ 
for Case 4.
%where $\tilde{\gamma} = \gamma + l_1 d - l_1 \delta$
%and $\tilde{i} = i + l_1 j - l_1 \delta$.
Then $0 < M \leq 1$.

\begin{lem}\label{lem2: d = 1}
If $d = 1$ and $\delta \neq T_k$ for any $k$, 
then 
\[
\left| \zeta (p^n(z)) \right| < \dfrac{C_1}{2^n R} \text{ and } |\eta (f^n(z,w))| < \dfrac{C_2}{( 2^n R )^M}
\]
on $U$
for some constants $C_1$ and $C_2$ and for large $R$.
\end{lem}

\begin{proof}
It is enough to consider Case 4.
There is a constant $A$ such that
$|\zeta| \leq A/|z|$.
Hence % it follows from  that 
$|\zeta (p^n)| \leq A/|p^n| \leq A/(2^n R)$ on $U$
by Lemma \ref{lem1: d = 1}.
Let $|w| = |z^{l_1}c|$ and $|z| = |tc^{l_2^{-1}}|$.
Then 
\[
|\eta(z,w)|
= \left| \sum \frac{b_{ij} z^{i} w^{j}}{z^{\gamma} w} \right| 
\leq \sum \frac{|b_{ij}|}{ |t|^{\tilde{\gamma} - \tilde{i}} |c|^{\tilde{d} - \tilde{j}} },
\]
where the sum is taken over all $(i,j) \neq (\gamma, d)$
such that $b_{ij} \neq 0$.
Recall that
$\tilde{\gamma} \geq \tilde{i}$ and $\tilde{d} \geq \tilde{j}$.
More precisely,
$\tilde{\gamma} - \tilde{i} \geq M$ if $\tilde{\gamma} > \tilde{i}$, and 
$\tilde{d} - \tilde{j} = d - j \geq 1$ if $\tilde{\gamma} = \tilde{i}$.
Hence there are constants $B$ and $C$ such that
$|\eta| \leq B/ |t|^M + C/ |c|$
and so $|\eta| \leq B/ |t|^M + C/ |c|^M$.
It then follows from Lemma \ref{lem1: d = 1} that 
$|\eta (f^n)| < (B + C)/(2^n R)^M$ on $U$.
\end{proof}

%Now we are ready to prove the uniform convergence of $\phi_n$.
Let $d = 1$ and $\delta \neq T_k$ for any $k$.
By Lemma \ref{lem2: d = 1},
\begin{align*}
|\Phi_{n+1}^2 - \Phi_n^2|
& \leq \frac{|Q(F^n) - Q_0(F^n)|}{d^{n+1}} 
+ \frac{\gamma_{n+1} |P(P^n) - P_0(P^n)|}{\delta^{n+1} d^{n+1}} \\
& \leq |\eta \circ \pi (F^n)| + \frac{\gamma}{\delta - 1} |\zeta \circ \pi (P^n)|
< \left( C_2 + \frac{\gamma}{\delta - 1} C_1 \right) \left( \frac{1}{2^n R} \right)^M
\end{align*}
on $\pi^{-1} (U)$.
Hence $\Phi_n$ converges uniformly to $\Phi$, % on $\pi^{-1} (U)$,
which induces the uniform convergence of $\phi_n$ to $\phi$. % on $U$.
Therefore,
$\phi$ is holomorphic on $U$ and
$\phi \sim id$ on $U$ as $R \to \infty$.

%%%%%%%%%%%%%%%%%%%%%%%%%%%%%%%%%%%%%%%%%%%%%%%%%%%%%%%%%
\subsection{Injectivity of $\phi$}

We prove that,
after enlarging $R$ if necessary,
the lift $F$ is injective on $\pi^{-1} (U)$.
Hence $F^n$, $\Phi_n$ and $\Phi$ are injective on the same region.
The injectivity of $\Phi$ induces that of $\phi$ % derives
because $\phi \sim id$ on $U$ as $R \to \infty$.
 
It is enough to consider Case 4. 
In that case,
$F$ is holomorphic on $\pi^{-1} (U)$,
where 
\[
\pi^{-1} (U)
= \left\{ l_1 \mathrm{Re} Z + \log R < \mathrm{Re} W < (l_1 + l_2) \mathrm{Re} Z - l_2 \log R \right\}.
\]
In particular,
$P$ is holomorphic and $|P- \delta Z| < \tilde{\varepsilon}$ on 
$\{ Z \ | \ \mathrm{Re} Z > ( 1 + l_2^{-1} ) \log R \}$.
% $H$,
% where
% \[
% H = \left\{ Z \ | \ \mathrm{Re} Z > \left( 1 + l_2^{-1} \right) \log R \right\}.
% \]
Rouch\'e's Theorem guarantees the injectivity of $P$ on a smaller region.
In fact,
the same argument as the proof of Proposition 6.1 in \cite{ueno SA} implies the following.
% see the proof of Proposition 6.1 in \cite{ueno SA} for a detail.

\begin{prop}\label{biholo of lift}
The function
$P$ is injective on 
\[
\left\{ Z \ \Big| \ \mathrm{Re} Z > \left( 1 + \dfrac{1}{l_2} \right) \log R +  \dfrac{2 \tilde{\varepsilon}}{\delta} \right\}.
\]
\end{prop}

% The proof of this proposition is almost the same as that of Proposition 6.1 in \cite{ueno SA}.
% The proof is almost the same as in \cite{ueno SA};
% see the proof of Proposition 6.1 in \cite{ueno SA}.
% One can prove this proposition
% by almost the same arguments as the proof of Proposition 6.1 in \cite{ueno SA}.

Let $Q_Z(W) = Q(Z,W)$ and $H_Z = H \cap (\{ Z \} \times \mathbb{C})$,
where
\[
H =
\left\{ l_1 \mathrm{Re} Z + \log R + \frac{2 \tilde{\varepsilon}}{d} < 
\mathrm{Re} W < (l_1 + l_2) \mathrm{Re} Z - l_2 \log R - \frac{2 \tilde{\varepsilon}}{d} \right\}.
\]
The same argument implies the injectivity of $Q_Z$ on $H_Z$.

\begin{prop}\label{biholo of lift}
The function
$Q_Z$ is injective on $H_Z$ for any fixed $Z$. 
\end{prop}

Note that $H \subset \left\{ \mathrm{Re} Z > \left( 1 + \dfrac{1}{l_2} \right) \log R + \dfrac{4 \tilde{\varepsilon}}{l_2 d} \right\}$
and let $C = \max \left\{ \dfrac{1}{d}, \dfrac{l_2}{2 \delta} \right\}$.

\begin{cor}
The maps $F$, $F^n$, $\Phi_n$ and $\Phi$ are injective on
\[
\left\{ l_1 \mathrm{Re} Z + \log R + 2 C \tilde{\varepsilon}
 < \mathrm{Re} W < (l_1 + l_2) \mathrm{Re} Z - l_2 \log R - 2 C \tilde{\varepsilon} \right\}.
\]
\end{cor}

As mentioned above,
the injectivity of $\Phi$ induces that of $\phi$. % derives

\begin{prop}\label{biholo of phi}
The B\"{o}ttcher coordinate $\phi$ is injective on 
\[
\left\{ 
(1 + \varepsilon)^{2C} R |z|^{l_1} < |w| < \frac{1}{(1 + \varepsilon)^{2C} R^{l_2}} |z|^{l_1 + l_2}  
\right\}.
\]
\end{prop}

%%%%%%%%%%%%%%%%%%%%%%%%%%%%%%%%%%%%%%%%%%%%%%%%%%%%%%%%%%%%%%%%%%%%%%%%%%%%
%%%%%%%%%%%%%%%%%%%%%%%%%%%%%%%%%%%%%%%%%%%%%%%%%%%%%%%%%%%%%%%%%%%%%%%%%%%%
\section{Extension of B\"{o}ttcher coordinates}

We extend the B\"{o}ttcher coordinate $\phi$ from $U$ to
a larger region in the union $A_{f}$ of all the preimages of $U$ under $f$.
Similar to the case of polynomials,
the obstruction is the critical set of $f$ % points
and we use analytic continuation in the proof.

Let $\psi$ be the inverse of $\phi$.
Because $\phi \sim id$ on $U$ as $R \to \infty$,
we may say that $\psi$ is biholomorphic on $U$.  
Our aim in this section is actually to extend $\psi$ from $U$ to a larger region $V$.
%%%
We first state our result and prove it in Section 8.1.
Although the proof is almost the same as in \cite{ueno SA},
we take $V$ as a more general region than that in \cite{ueno SA}.
We then calculate the union $A_{f_0}$ of all the preimages of $U$ 
under the monomial map $f_0$ in Section 8.2 and
provide two concrete examples of $V$ with four parameters in Section 8.3.

%%%%%%%%%%%%%%%%%%%%%%%%%%%%%%%%%%%%%%%%%%%%%%%%%%%%%%%%%
\subsection{Statement and Proof}

Let $|\phi| = (|\phi_1|, |\phi_2|)$, 
which extends to a continuous map from $A_f$ to $\mathbb{R}^{2}$
via $(f_0 |_{\mathbb{R}^{2}})^{-n} \circ |\phi| \circ f^n$.
We require $V$ to be a connected, simply connected %  the region
Reinhardt domain and included in $A_{f_0}$.
Moreover,
we require that $V \cap (\{ z \} \times \mathbb{C})$ is connected for any $z$.
For simplicity,
we also require $V$ to include $U$.

\begin{thm}\label{extension}
Let $V$ be a region as above. 
If $f$ has no critical points in $|\phi|^{-1} (V \cap \mathbb{R}^{2})$,
then $\psi$ extends by analytic continuation
to a biholomorphic map on $V$.
\end{thm}

\begin{proof} 
Using the same arguments as the proof of Theorem 6 in \cite{ueno poly},
one can show that 
$\psi$ extends to a holomorphic map on $V$ by analytic continuation.

We show that $\psi$ is homeomorphism on $V$
by adapting the arguments of the proof of Theorem 9.5 in \cite{ueno SA}
to the case of polynomial skew products.
By the constriction of $\psi$, 
it is locally one-to-one,
and the set of all pairs $x_1 = (z_1,w_1) \neq x_2 = (z_2,w_2)$
with $\psi (x_1) = \psi (x_2)$ forms a closed subset of $V \times V$.
If $\psi (x_1) = \psi (x_2)$,
then $|z_1|=|z_2|$ and $|w_1|=|w_2|$ 
because $|\phi \circ \psi|=|id|$.
Assuming that there were such a pair with $\psi (x_1) = \psi (x_2)$,
we derive a contradiction.
There are two cases:
the maximum of $|z_1|$ exists or not.
First,
assume that the maximum exists.
Since $\psi$ is an open map,
for any $x'_1$ sufficiently close to $x_1$, 
we can choose $x'_2$ close to $x_2$ 
with $\psi (x'_1) = \psi (x'_2)$.
In particular,
we can choose $x'_j$ with $|z'_j| > |z_j|$,
which contradicts the choice of $z_j$.
Next,
assume that the maximum does not exist.
Then there is a pair with $|z_1| = |z_2| > R^{1+ l_2^{-1}}$.
Fix such $z_1$.
For Cases 1 and 2,
the intersection of $V - U$ and
the fiber $\{ z_1 \} \times \mathbb{C}$ is an annulus, 
% the fiber at $z_1$ is an annulus, 
and we can choose $|w_1|$ as maximal.
Using the same argument as above to the fibers
$\{ z_1 \} \times \mathbb{C}$ and $\{ z_2 \} \times \mathbb{C}$,
we can choose $x'_1 = (z_1,w'_1)$ and $x'_2 = (z_2,w'_2)$
so that $\psi (x'_1) = \psi (x'_2)$ and $|w'_j| > |w_j|$,
which contradicts the choice of $w_j$.
For Cases 3 and 4,
the intersection may consist of two annuli. % annuluses.
For this case,
we can choose $|w_1|$ as minimal in the outer annulus
or as maximal in the inner annulus,
which derives a contradiction by the same argument as above.
\end{proof}

%%%%%%%%%%%%%%%%%%%%%%%%%%%%%%%%%%%%%%%%%%%%%%%%%%%%%%%%%
\subsection{Monomial maps}

Let $f_0 (z,w) = (z^{\delta}, z^{\gamma} w^{d})$,
where $\delta \geq 2$, $\gamma \geq 0$,
$d \geq 1$ and $\gamma + d \geq 2$.
% we assume that the coefficients are both $1$ for simplicity.
% We calculate the union of the all preimages of $U$ under $f_0$.
Let $R >1$ and
\[
A_{f_0} = A_{f_0} (U) = \bigcup_{n \geq 0} f_0^{-n} (U),
\]
which is included in the divergent region for $f_0$.
The affine function 
\[
T(l) = \frac{\delta l - \gamma}{d}
\]
plays a central role to calculate $A_{f_0}$.
% Inductively,
Since $f_0^n (z,w) = (z^{\delta^n}, z^{\gamma_n} w^{d^n})$
and $T^n(l) = (\delta^n l - \gamma_n)/d^n$, 
where $\gamma_n = \sum_{j=1}^{n} \delta^{n-j} d^{j-1} \gamma$,
% one can show that 
the preimage $f_0^{-n} (U)$ is equal to
\begin{enumerate}
\item $\{ |z| > R^{{1/\delta}^n}, |w| > R^{1/d^{n}} |z|^{T^n(0)} \}$ for Case 1,
\item $\{ |z| > R^{1/{\delta}^{n}}, |w| > R^{1/d^{n}} |z|^{T^n(l_1)} \}$ for Case 2,  
\item $\{ R^{1/d^{n}} |z|^{T^n(0)} < |w| < R^{- l_2 /d^{n}} |z|^{T^n(l_2)} \}$ for Case 3, and
\item $\{ R^{1/d^{n}} |z|^{T^n(l_1)} < |w| < R^{- l_2 /d^{n}} |z|^{T^n(l_1 + l_2)} \}$ for Case 4.
\end{enumerate}
If $\delta \neq d$,
then 
\[
T(l) = \frac{\delta}{d} (l - \alpha_0) + \alpha_0
\text{ and so }
T^n (l) = \left( \frac{\delta}{d} \right)^n (l - \alpha_0) + \alpha_0,
\]
where $\alpha_0 = \gamma /(\delta - d)$.
Therefore,
for Case 1, the region $A_{f_0}$ is equal to
\begin{enumerate}
\item $\{ |z| > 1, w \neq 0 \}$ if $\delta \geq d$ and $\gamma > 0$,
\item $\{ |z| > 1, |z^{- \alpha_0} w| > 1 \}$ if $\delta < d$ and $\gamma > 0$, where $\alpha_0 < 0$, or
\item $\{ |z| > 1, |w| > 1 \}$ if $\gamma = 0$.
\end{enumerate}
%%%
For Case 2, the inequality $\delta \leq T_1$ holds and $A_{f_0}$ is equal to
\begin{enumerate}
\item $\{ |z| > 1, w \neq 0 \}$ if $T_{1} > \delta \geq d$ and $\gamma > 0$, 
\item $\{ |z| > 1, |w| > |z|^{\alpha_0} \}$ if $T_{s-1} = \delta > d \geq 2$ and $\gamma > 0$, 
\item $\{ |z| > 1, |z^{- \alpha_0} w| > 1 \}$ if $\delta < d$ and $\gamma > 0$, where $\alpha_0 < 0$,
\item $\{ |z| > 1, |w| > 1 \}$ if $\delta < d$ and $\gamma = 0$, or
\item $\{ |z| > 1, |w| > |z|^{l_1} \}$ if $\delta = d$ and $\gamma = 0$.
\end{enumerate}
%%%
For Case 3, the inequalities $\delta \geq T_{s-1} > d$ and $\gamma > 0$ hold and $A_{f_0}$ is equal to
\begin{enumerate}
\item $\{ |z| < 1, w \neq 0 \}$ if $\delta > T_{s-1}$, or
\item $\{ |z| > 1, 0 < |w| < |z|^{\alpha_0} \}$ if $\delta = T_{s-1}$ and $d \geq 2$.
\end{enumerate}
%%%
For Case 4, 
the inequalities $T_{k} \geq \delta \geq T_{k-1} > d$ and $\gamma > 0$ hold and $A_{f_0}$ is equal to
\begin{enumerate}
\item $\{ |z| > 1, w \neq 0 \}$ if $T_{k-1} < \delta < T_k$, 
\item $\{ |z| > 1, |w| > |z|^{\alpha_0} \}$ if $\delta = T_k$ and $d \geq 2$, or
\item $\{ |z| > 1, 0 < |w| < |z|^{\alpha_0} \}$ if $\delta = T_{k-1}$ and $d \geq 2$.
\end{enumerate}
Note that we only display the cases that appear in the main theorem;
we do not have the case $\delta > d$ and $\gamma = 0$ in Case 2, 
the case $\gamma = 0$ in Cases 3 and 4, and
the case $d = 1$ and $\delta = T_j$ for some $j$. 

\begin{rem}
For Case 1,
the region $A_{f_0}$ does not change even if we replace $U$ to 
the larger region described in Remark \ref{larger regions for case 1}.
\end{rem}

%%%%%%%%%%%%%%%%%%%%%%%%%%%%%%%%%%%%%%%%%%%%%%%%%%%%%%%%%
\subsection{Examples of $V$}

The following two concrete examples of $V$ satisfy
all the assumptions in Theorem {\rmfamily \ref{extension}}.

\begin{example}
Let $V = \{ |z| > r_1, |w| > r_2 |z|^{a_1} \}$ for Cases 1 and 2 and
let $V = \{ r_2 |z|^{a_1} < |w| < r_1^{-l_2} |z|^{a_2} \}$ for Cases 3 and 4,
% where $1 \leq r_1$, $r_2 \leq R$ and
where $1 \leq r_1 \leq R$, $1 \leq r_2 \leq R$ and
$- \infty \leq a_1 \leq l_1 < l_1 + l_2 \leq a_2 \leq \infty$.
\end{example}

\begin{example}
Let $V = \{ r_2 |z|^{a_1} < |w| < r_1^{-a_2} |z|^{a_2} \}$,
where $1 \leq r_1 \leq R^{l_2/(l_1 + l_2)}$, $1 \leq r_2 \leq R$ and
$- \infty \leq a_1 \leq l_1 < l_1 + l_2 \leq a_2 \leq \infty$.
\end{example}

Against the case of polynomials,
in which we only have one direction $r_1$ of extension,
here we have four directions $r_1$, $r_2$, $a_1$ and $a_2$
for the case of polynomial skew products.

For both examples,
V coincides with $U$ and realizes all the types of $A_{f_0}$
for suitable choices of the four parameters.
%%%
We remark that 
we do not need to require $V$ to include $U$;
Theorem {\rmfamily \ref{extension}} holds 
on $V \cup U$ if $V \cap U \neq \emptyset$.
Hence we may widen the ranges of the parameters in the examples above to
$1 \leq r_1$, $1 \leq r_2$, % $1 \leq r_1$, $r_2$,
$- \infty \leq a_1 < \infty$ and $0 < a_2 \leq \infty$.

%%%%%%%%%%%%%%%%%%%%%%%%%%%%%%%%%%%%%%%%%%%%%%%%%%%%%%%%%%%%%%%%%%%%%%%%%%%%
%%%%%%%%%%%%%%%%%%%%%%%%%%%%%%%%%%%%%%%%%%%%%%%%%%%%%%%%%%%%%%%%%%%%%%%%%%%%
\section{Other changes of coordinate}

% Finally,
We provide two other changes of coordinate in the last section,
which are derived from the B\"{o}ttcher coordinate $\phi$ and
already appeared in Corollaries 1 and 11 in \cite{ueno poly}.
Although the relation between $\phi_2$ and $\chi$ % the functions  
for Cases 3 and 4 is less clear than that for Cases 1 and 2,
we obtain the same conclusion with the condition $d \geq 2$
for the former change of coordinate.
% On the other hand,
% we do not need the condition $d \geq 2$
% for the latter changes of coordinate.
On the other hand,
we have the latter change of coordinate
even for the case $d = 1$.

Let $b(z)$ be the coefficient of $w^d$ in $q$.
% which contains $z^{\gamma}$.
Then $b(z) = b_{\gamma d} z^{\gamma} \{ 1 + \tilde{\zeta} (z) \}$ 
and $q(z,w) = b(z) w^d \{ 1 + \tilde{\eta} (z,w) \}$ on $U$,
where $\tilde{\zeta}$, $\tilde{\eta} \to 0$ on $U$ as $R \to \infty$, and so
the second component of $f^n$ is written as 
$B_n(z) w^{d^n} \{ 1 + \tilde{\eta}_n (z,w) \}$ on $U$,
where $B_n (z) = \prod_{j = 0}^{n-1} (b(p^j(z)))^{d^{n-1-j}}$ 
and $\tilde{\eta}_n \to 0$ on $U$ as $R \to \infty$.
Therefore,
\[
\phi_2 (z,w) = \lim_{n \to \infty} 
\sqrt[d^n]{\dfrac{B_n(z) w^{d^n} \{ 1 + \tilde{\eta}_n (z,w) \} }{(p^n(z))^{\gamma_n/ \delta^n}}}.
\]
We define
\[
\chi (z) = \lim_{n \to \infty} 
\sqrt[d^n]{\dfrac{B_n(z)}{(p^n(z))^{\gamma_n/ \delta^n}}}.
\]
We can show that 
$\chi$ is well defined and holomorphic on $\{ |z| > R \}$,
$\chi \to 1$ as $z \to \infty$ and
$\chi \circ p = b^{-1} \cdot \varphi_p^{\gamma} \cdot \chi^d$
if $d \geq 2$
by the same arguments as the proofs of Lemmas 8 and 9 in \cite{ueno poly}.
Let
\[
\tilde{\phi}_2 (z, w) = \dfrac{\phi_2 (z, w)}{\chi (z)}.
\]
%Then we have the following corollary.

\begin{cor}\label{another coord}
If $d \geq 2$, 
then the biholomorphic map $(z, \tilde{\phi}_2 (z, w))$ defined on $U$ 
conjugates $f$ to $(z,w) \to (p(z), b(z) w^d)$.
\end{cor}

In addition,
if $\delta \neq d$ and $\alpha_0$ is integer, 
then let
\[
\phi_z^{\alpha_0} (w)
= \dfrac{\phi_2 (z, w)}{(\phi_1 (z))^{\alpha_0}}.
\]  

\begin{cor}
Let $d \geq 2$ or let $d = 1$ and $\delta \neq T_k$ for any $k$. 
If $\delta \neq d$ and $\alpha_0$ is integer, 
then the biholomorphic maps $(\phi_1 (z), \phi_z^{\alpha_0} (w))$ 
and $(z, \phi_z^{\alpha_0} (w))$ defined on $U$ conjugate $f$ 
to $(z,w) \to (z^{\delta}, w^d)$ and $(z,w) \to (p(z), w^d)$,
respectively.
\end{cor}

%%%%%%%%%%%%%%%%%%%%%%%%%%%%%%%%%%%%%%%%%%%%%%%%%%%%%%%%%%%%%%%%%%%%%%%%%%%%
%%%%%%%%%%%%%%%%%%%%%%%%%%%%%%%%%%%%%%%%%%%%%%%%%%%%%%%%%%%%%%%%%%%%%%%%%%%%

%%%%%%%%%%%%%%%%%%%%%%%%%%%%%%%%%%%%%%%%%%%%%%%%%%%%%%%%%%%%%%%%%%%%%%%%%%%%%
%%%%%%%%%%%%%%%%%%%%%%%%%%%%%%%%%%%%%%%%%%%%%%%%%%%%%%%%%%%%%%%%%%%%%%%%%%%%%
%%%%%%%%%%%%%%%%%%%%%%%%%%%%%%%%%%%%%%%%%%%%%%%%%%%%%%%%%%%%%%%%%%%%%%%%%%%%%
\end{document}